\documentclass[12pt,reqno]{amsart}

\usepackage{fullpage}
\usepackage{amssymb,latexsym,mathrsfs,amsrefs,mathtools}
\usepackage{bm}
\usepackage{bbm}
\usepackage{bigints}
\usepackage{relsize}
\usepackage{url}
\usepackage{enumerate}[1.)]
\numberwithin{equation}{section}
\usepackage[colorlinks]{hyperref}

\newtheorem{theorem}{Theorem}[section]
\newtheorem{lemma}[theorem]{Lemma}
\newtheorem{proposition}[theorem]{Proposition}
\newtheorem{corollary}[theorem]{Corollary}
\newtheorem{definition}[theorem]{Definition}
\newtheorem{remark}{Remark}[section]

\title{Sub-convexity problem for Rankin-Selberg $L$-functions}
\author{Chandrasekhar Raju}

\begin{document}

\begin{abstract}
We establish a sub-convexity estimate for Rankin-Selberg $L$-functions in the combined level aspect, using the circle method. If $p$ and $q$ are distinct prime numbers, $f$ and $g$ are non-exceptional newforms (modular or Maass) for the congruence subgroups $\Gamma_0(p)$ and $\Gamma_0(q)$ (resp) with trivial nebentypus, then for all $\epsilon >0$ we show that there exists an $A >0$ such that
$$
L\left(\frac{1}{2}+it, f \times g \right) \ll_{\epsilon,\mu_f, \mu_g}(1+|t|)^A \frac{(pq)^{1/2+\epsilon}}{\max\{p,q \}^{\frac{1}{64}}}.
$$
The dependence on $\mu_f$ and $\mu_g$, the parameters at infinity for $f$ and $g$ respectively, is polynomial. Further, if $p$ is fixed and $q \rightarrow \infty$, we improve this to
$$
L\left(\frac{1}{2}+it, f \times g \right) \ll_{\epsilon,\mu_f,\mu_g}(p(1+|t|))^Aq^{\frac{1}{2}-\frac{1-2\theta}{27+28\theta}+\epsilon} ,
$$
where $\theta$ is the exponent towards Ramanujan-conjecture for cuspidal automorphic forms. Unconditionally, we can take $\theta = 7/64$. This improves all previously known sub-convexity estimates in this case.

\end{abstract}
\maketitle

\section{Introduction}

Understanding the behaviour of automorphic $L$-functions in the critical strip is an important problem in modern analytic number theory. A problem which has received a lot of attention in this area is the sub-convexity problem. Let $L(f,s)$ be an automorphic $L$-function and $s \in \mathbb{C}$ be such that $\mathfrak{Re}(s) = \frac{1}{2}$, then the following inequality is called the convexity bound:
\begin{equation*}
L(f,s) \ll Q(f,s)^{1/4+\epsilon},
\end{equation*}
where $Q(f,s)$ is the analytic conductor of the $L$-functions at $s$. We refer the reader to \cite{iwaniec2000perspectives} for the definition of analytic conductor. Obtaining any positive saving in the exponent $1/4$, is typically called the sub-convexity problem. The generalized Lindelof-hypothesis, which is a consequence of the generalized Riemann-Hypothesis, asserts that for any $\epsilon >0$
\begin{equation*}
L(f,s) \ll Q(f,s)^{\epsilon}.
\end{equation*}
The convexity bound, can be viewed as the trivial bound for the $L$-function on the half line $\mathfrak{Re}(s) = \frac{1}{2}$. One way to view a sub-convexity estimate is as progress towards generalized Riemann-hypothesis. Perhaps more concretely, the sub-convexity problem for various families of $L$-functions connects to various equidistribution problems. One example of such a connection is the relation between Quantum Unique Ergodicity (QUE) and sub-convexity estimates for the symmetric square $L$-functions (see \cite{sarnak2001}). We refer the reader to \cite{friedlander1995survey}, \cite{iwaniec2000perspectives}, \cite{Michel2003survey}, and \cite{michel2006parkcity} for a survey of the other applications.

The aim of this work is twofold. Firstly, we exhibit a sub-convexity estimate for Rankin-Selberg $L$-functions in the combined level aspect of both the automorphic forms, as long as the conductor of the $L$-function doesn't drop.
Secondly, as we use the circle method instead of an amplified second moment to establish this result, we are able to completely bypass the use of the Kuznetsov trace formula. This makes the proof considerably less technical. Moreover, we are able to improve the known results considerably (see Theorem \ref{main theorem when f is fixed}) by avoiding the technical complications and directly cutting to the heart of the matter.

In the case that $f$ is a $GL(1)$ automorphic form, the sub-convexity problem was solved due to the work of Weyl \cite{weyl1921} and Burgess \cite{burgess1963character}. Iwaniec introduced the amplification method, in \cite{iwaniec1992spectral}, to prove sub-convexity estimates for $GL(2)$ $L$-functions in the spectral aspect. Following this, Duke, Friedlander and Iwaniec established sub-convexity bounds in the level aspect for $GL(2)$ $L$-functions in a series of papers culminating in \cite{duke2002subconvexity}. The problem becomes significantly harder to tackle when we head to $GL(3)$ $L$-functions.

Munshi \cite{munshi2014circle} proved a hybrid sub-convexity bounds for $GL(2)$ $L$-functions twisted by a Dirichlet character in the critical strip by a very different argument. If $f$ is a modular form for $PSL_2(\mathbb{Z})$ and $\chi$ is a character mod $q$, he shows that
\begin{equation}
L(1/2+it,f \times \chi) \ll_{\epsilon} (q(3+|t|))^{1/2-1/18+\epsilon}.
\end{equation}
The main novelty in the argument is to directly separate the oscillation of $f$ and $\chi$ in an approximate functional equation for $L(f \times \chi,1/2+it)$, using Jutila's circle method. Using a set of factorable moduli in the circle method, he obtains some extra cancellation to break the convexity bound. He has advanced the use of circle method to obtain sub-convexity bounds in a series of papers. Notably, he obtained the first sub-convexity bound for the value of non self dual $GL(3)$ cusp form in $t$-aspect \cite{munshi2015circle} (The self-dual case was already known due to work of Li \cite{Li2011}). He also came up with the ``$GL(2)$ circle method" to establish a sub-convexity bound for $GL(3)$ cusp forms twisted by Dirichlet character \cite{munshi2015annals}. Holowinsky and Nelson have simplified the latter result's proof considerably in \cite{holowinsky2018subconvex}. Munshi also used the $GL(2)$ circle method to re-establish a Burgess type bound for character twists of $GL(2)$ automorphic forms (including the Eisenstein series, which recovers Burgess's original bound for Dirichlet $L$-function) in \cite{munshi2017note}. In a similar vein Aggarwal, Holowinsky, Lin and Sun simplified Munshi's work in \cite{aggarwal2018burgessbound}. Munshi has also used the circle method to obtain sub-convexity bounds for the symmetric square $L$-function of holomorphic $GL(2)$ cusp forms, in the level aspect \cite{munshi2017subconvexity}.

Before we begin, we set up some basic notation. We say that a cusp form $f$ is ``non-exceptional", if either $f$ is  holomorphic or the eigenvalue $\lambda_f$ of $f$ under the Laplacian ($-\Delta$) satisfies, $\lambda_f \geq \frac{1}{4}$. Selberg's eigenvalue conjecture asserts that there are no exceptional forms.

 Let $f$ and $g$ be primitive cuspidal newforms (not necessarily holomorphic) for $\Gamma_0(p)$ and $\Gamma_0(q)$, with nebentypus $\chi_f$ and $\chi_g$ respectively. These are eigenforms of suitably normalized Hecke operators $\{T_n\}_{n\geq 1}$ with eigenvalues $\lambda_f(n)$ and $\lambda_g(n)$ respectively. For all primes $l$, these eigenvalues for $f$ can be written as
$$
\lambda_f(l) = \alpha_{f,1}(l) +\alpha_{f,2}(l)\text{, }{} \alpha_{f,1}(l)\alpha_{f,2}(l) = \chi_f(l)
$$
and similarly for $g$. The Rankin-Selberg $L$-function is defined by
$$
L(f \times g,s) = L(\chi_f\chi_g,2s)\sum_{n=1}^\infty{\frac{\lambda_f(n)\lambda_g(n)}{n^s}}.
$$
If $(p,q)=1$, then we have the following Euler product for $L(f \times g,s)$:
$$
L(f \times g,s) = \prod_{l \text{ prime} }\prod_{i,j =1,2}{\left(1- \frac{\alpha_{f,i}\alpha_{g,j}}{l^s} \right)^{-1}}.
$$
Moreover, the equality above holds in general even when $p$ and $q$ are not coprime except for  finitely many Euler factors at the primes $l$ dividing $(p,q)$.
The arithmetic conductor of this $L$-function, $Q(f\times g)$, satisfies \cite{harcos2006subconvexity}
\begin{equation}\label{arithmetic conductor}
\frac{(pq)^2}{(p,q)^4} \leq Q(f \times g) \leq \frac{(pq)^2}{(p,q)}.
\end{equation}
Thus if $(p,q) =1$, the convexity estimate for the $L$-function is 
$$
L(f \times g, 1/2+it) \ll_{t,\epsilon} Q(f\times g)^{1/4+\epsilon}=(pq)^{1/2+\epsilon}.
$$

In the case that $f$ is fixed and we let $q$ vary, sub-convexity estimates are known due to the work of Michel, Kowalski, Vanderkam, and Harcos (\cite{kowalski2000mollification}, \cite{kowalski2002rankin}, \cite{michel2004subconvexity}, and \cite{harcos2006subconvexity}).  Using the amplification method and the Kuznetsov trace formula (assuming $\chi_f\chi_g$ is not trivial), they established that \cite{harcos2006subconvexity}
$$
L(f \times g,1/2) \ll_{f,\epsilon} q^{1/2-1/2648+\epsilon}.
$$
Better exponents are known in particular cases. 
In particular, if $f$ is a fixed non-exceptional cuspidal automorphic form and $g$ has trivial central character,  then Kowalski, Michel, and Vanderkam \cite{kowalski2002rankin} showed that
\begin{equation}\label{kowalski,michel,vanderkam bound}
L(f \times g,1/2) \ll_{f,\epsilon} q^{1/2-1/80+\epsilon}.
\end{equation}

In this work, we tackle the case when both $p$ and $q$ vary simultaneously such that $(p,q)=1$ and $\chi_f, \chi_g$ are both the trivial character. This question has been treated in the works of Michel-Ramakrishnan \cite{michel2012consequences}, Feigon-Whitehouse \cite{feigon2009averages}, Nelson \cite{nelson2013stable}, and Holowinsky-Templier \cite{holowinsky2014first} in situations where positivity of the central value is known. Holowinsky and Munshi \cite{holowinsky2012level} obtained a sub-convexity bound for this problem as long as $p \leq q^{\eta}$, with $\eta = \frac{2}{21}$. Hou and Zhang extended this to $\eta = \frac{2}{15}$ \cite{hou2017hybrid}. Assuming that the form with the smaller level is holomorphic Zhilin Ye \cite{ye2014second} proves a sub-convexity bound for all $\eta$. It has been indicated that $\delta =1/801$ (in the notation of Theorem \ref{main theorem} below) is admissible. We remove the holomorphicity assumption and improve the sub-convexity exponent considerably by a different method.
 
\begin{theorem}\label{main theorem}
Let $f$ and $g$ be primitive non-exceptional cuspidal newforms (not necessarily holomorphic) of prime levels $p$ and $q$ respectively, with trivial nebentypus. If $\delta=1/64$, $p \neq q$ and $\mathfrak{Re} s =1/2$, then for any $\epsilon >0$
$$
L(f \times g,s) \ll_{s,\epsilon} \frac{(pq)^{1/2+\epsilon}}{\max\{p,q\}^\delta },
$$
with the implied constant depending polynomially on $|s|$, $\epsilon$, and spectral parameters of $f$ and $g$ at infinity.
\end{theorem}

\begin{remark}
We can also treat exceptional forms, at the cost of a smaller exponent of sub-convexity. The primality of $p$ and $q$, can be replaced by the condition that $p$ and $q$ are coprime. We have avoided carrying this out to simplify the exposition and keep the ideas clear. Our method also works in the case $(p,q)>1$, as long as the arithmetic conductor of $f \times g$ doesn't drop much. The hypothesis $c(f \times g) \gg_{\epsilon} (pq)^{1+\epsilon}$ is sufficient. This does not include $f \times f$, which is related to the sub-convexity of the symmetric square $L$-function $L(\text{sym}^2 f,s)$. 
\end{remark}

If either one of $p$ or $q$ is very small (i.e $p$ is bounded by a small power of $q$, say $q^{1/1000}$), then this problem can be solved using the methods of \cite{kowalski2002rankin}. The most interesting and hardest case of the theorem is when both $p$ and $q$ are both large.  In this case, the crux of the proof is the solution to a shifted convolution problem \eqref{definition of SCP}, where the shifts are multiples of levels of the modular form. 
Munshi encounters a very similar problem in his work on the symmetric square $L$-function \cite{munshi2017subconvexity}. We state this below in the form of a theorem, as this might be of independent interest in connection to other problems.

\begin{theorem}\label{SCP in ab coprime to p}
Let $p$ be a prime number or $p=1$. Let  $a,b,c,d$ be integers such that $a$ and $b$ are co-prime to $p$.  Further, let $f,g$ be non-exceptional cuspidal newforms (modular or Maass) of level p and trivial nebentypus. For any $M_1,M_2,K_1,K_2 \geq 1$, we claim the following upper bounds for the shifted convolution sum $S_{f,g}$:
\begin{equation}\label{bound for SCP rankin-selberg}
S_{f,g}(a,b,c,d,M_1,M_2) \ll p^\epsilon \min \{(M_1M_2)^{1/2},(M_1M_2)^\theta X\},
\end{equation}
where $S_{f,g}$  and $X$ have been defined in \ref{definition of SCP} and equation \eqref{defintion of X} respectively. Here $\theta$ is the exponent towards Ramanujan conjecture for $f$ and $g$. 
If the shift $ad-bc$ is non-zero, then
\begin{equation}\label{bound for SCP non-zero shift}
S_{f,g}(a,b,c,d,M_1,M_2) \ll (K_1K_2)^{3/2}\sqrt{ab}p^{3/4}X^{3/4}.
 \end{equation}
Furthermore, if the shift $ad-bc$ is a non-zero multiple of $p$, then
\begin{equation}\label{bound for SCP shift multiple of level}
S_{f,g}(a,b,c,d,M_1,M_2) \ll (abpM_1M_2)^\epsilon (K_1K_2)^{3/2}\sqrt{ab}p^{1/4}X^{3/4} .
 \end{equation}

\end{theorem}

We believe that it is possible to improve the upper bound in \eqref{bound for SCP shift multiple of level} to 
$ O_{K_1,K_2,\epsilon}(\sqrt{ab}p^{1/4}X^{1/2})$ using spectral theory. If the level of $f$ is fixed we obtain a better exponent in Theorem \ref{main theorem}, using such an improvement known due to the work of Blomer \cite{blomer2004shifted}. This improves the previously known sub-convexity bounds due to Kowalski, Michel, and Vanderkam \eqref{kowalski,michel,vanderkam bound}.

\begin{theorem}\label{main theorem when f is fixed}
Let $f$ and $g$ be primitive non-exceptional cuspidal newforms (not necessarily holomorphic) of prime levels $p$ and $q$ respectively, with trivial nebentypus. If $\delta = \frac{1- 2\theta}{27+28\theta}$, then for any $s \in \mathbb{C}$ with $\mathfrak{Re} s =1/2$ and any $\epsilon>0$
$$
L(f \times g,s) \ll_{s,\epsilon,p} q^{\frac{1}{2}-\delta+\epsilon},
$$
with  the implied constant depending polynomially on $|s|$, $\epsilon$, $p$, and parameters of $f$ and $g$ at infinity. Here $\theta$ is the exponent towards Ramanujan conjecture for cuspidal automorphic forms on $GL(2)$. Unconditionally, we can take $\theta = 7/64$. This gives $\delta = 0.02598\dots$ .
\end{theorem}
Though Theorem \ref{main theorem when f is fixed} has been stated for $f$ having trivial central character, one can go through the proof and check that, with minor modifications, the proof works even if $f$ has a non-trivial central character. But as this is not possible in Theorem \ref{main theorem}, we have chosen not to write this down separately. 
However, we are unable to handle $f$ being an Eisenstein series. Hence, our result does not recover a sub-convexity estimate for $GL(2)$ $L$-functions in the level aspect \cite{duke2002subconvexity}. The issue here is the presence of main terms, and a similar issue demanding a delicate cancellation argument arose in \cite{duke2002subconvexity}. It would be interesting to resolve this case using the circle method.

As a Corollary to Theorem \ref{main theorem when f is fixed}, we improve the bounds obtained by Kowalski, Michel, and Vanderkam \cite{kowalski2002rankin} for the problem of distinguishing modular forms based on their first Fourier coefficients.
\begin{corollary}
Let $f$ be a primitive cusp form and $\delta =  \frac{1- 2\theta}{27+28\theta} $ as in Theorem \ref{main theorem when f is fixed} of prime level $p$ and $\epsilon >0$. There exists a constant $C =C(f,\epsilon)$ such that for any primitive cuspidal new form $g$ of prime level $q$, there exists $n \leq q^{1 - 2\delta +\epsilon}$ such that
$$
\lambda_f(n) \neq \lambda_g(n).
$$
\end{corollary}
\begin{proof}
The proof is identical to the proof of Corollary 1.3 \cite{kowalski2002rankin}. Use Theorem \ref{main theorem when f is fixed} in place of \cite[Theorem 1.1]{kowalski2002rankin} in the proof.
\end{proof}

We briefly review the facts we need about $GL(2)$ automorphic forms in the next section. We do not need the $GL(2)$ trace formula. We will say a few more comments on this point in Section \ref{outline of proof} (see \eqref{second moment}), where we also briefly sketch the outline of the proof. In Section \ref{voronoi transformations and simplifications}, we carry out the initial transformations leading us to the shifted convolution problem. We obtain upper bounds for the shifted convolution problem in Section \ref{shifted convolution problem}. We end the paper by combining the bounds obtained before to prove Theorem \ref{main theorem} and Theorem \ref{main theorem when f is fixed} in Section \ref{proof of main proposition}.

\subsection*{Acknowledgment}
The author thanks Kannan Soundararajan and Valentin Blomer for a careful reading and helpful comments. The author is also grateful to Roman Holowinsky and Paul Nelson for their encouragement. The author is supported by B.C. and E.J. Eaves Stanford Graduate Fellowship.

\section{Review of automorphic forms}

We state the facts we need briefly in this section. We refer the reader to \cite{harcos2006subconvexity}, \cite{michel2004subconvexity} for a complete account.
\subsection{Voronoi summation}
We recall the Voronoi summation formula from  \cite[Theorem A.4]{kowalski2002rankin}.
\begin{lemma}\label{voronoi summation}
Let $D$ be a positive integer,  $\chi_D$ be a character of modulus $D$. Further, let $g$ be either a holomorphic form of weight $k_g \geq 2$ or a Maass form of eigenvalue $\lambda_g$, level $D$, and central character $\chi_D$. For $(a,c)=1$, set  $D_1= (c,D)$ and $D_2= D/D_1$ and assume that $(D1,D2)=1$, so that $\chi_D = \chi_{D_1}\chi_{D2}$ is the unique factorization of $\chi_D$  into characters of modulus $D_1$ and $D_2$. For $F \in C^\infty(\mathbb{R}^+)$,  a smooth function vanishing in a neighborhood of $0$ and decreasing rapidly,
\begin{align*}
&\sum_{n \geq 1}{\lambda_g(n)e\left(\frac{an}{c} \right)F(n)}\\
&= \frac{\chi_{D_1}(\overline{a})\chi_{D_2}(-c)\eta_g(D_2)}{c\sqrt{D_2}}\sum_{n \geq 1}{\lambda_{g_{D_2}}(n)e\left( \frac{-n\overline{aD_2}}{c}\right) \int_0^\infty{F(x)J^+_g\left(\frac{4\pi}{c}\sqrt{\frac{nx}{D_2}}\right) dx}}\\
&+ \frac{\chi_{D_1}(\overline{a})\chi_{D_2}(c)\eta_g(D_2)}{c\sqrt{D_2}}\sum_{n \geq 1}{\lambda_{g_{D_2}}(n)e\left( \frac{n\overline{aD_2}}{c}\right) \int_0^\infty{F(x)J^-_g\left(\frac{4\pi}{c}\sqrt{\frac{nx}{D_2}}\right) dx}}.
\end{align*}

In this formula
\begin{itemize}
\item
$\eta_g(D_2)$ is the pseudo-eigenvalue of the Atkin-Lehner operator $W_{D_2}$; if $\lambda_g(D_2) \neq 0$, it equals
$$
\eta_g(D_2) = \frac{G(\chi_{D_2})}{\lambda_g(D_2)\sqrt{D_2}};
$$
\item
if $g$ is holomorphic of weight $k_g$, then
$$
J^+_g(x) = 2\pi i^{k_g}J_{k_g-1}(x)\text{, } J^-_g(x)=0;
$$
\item
if $g$ is Maass form with $(\Delta + \lambda)g=0$ and let $r$ satisfy $\lambda_g = \left(\frac{1}{2} +ir \right) \left( \frac{1}{2} -ir\right)$, and let $\epsilon_g$ be the eigenvalue of $g$ under the reflection operator. Then
$$
J^+_g(x)= \frac{-\pi}{\sin (\pi i r) }\left(J_{2ir}(x) - J_{-2ir} (x) \right), \text{ } J^-_g(x) = \epsilon_g 4\cosh (\pi r)K_{2ir}(x);
$$
\item
if $r=0$,
$$
J^+_g(x) = -2\pi Y_0(x),\text{ } J^-_g(x)= \epsilon_g4K_0(x).
$$
\end{itemize}	
\end{lemma}

\begin{remark}
We shall use this Lemma only when $\chi_D$ is the trivial character i.e $g$ has trivial nebentypus. In this case $W_{D_2}$ is an endomorphism on the space of cusp forms of level $D_2$, $|\eta_g(D_2)| =1$ and $g_{D_2} =g$ \cite[Proposition A.1]{kowalski2002rankin}. 
\end{remark}

We need to understand the behaviour of the integral transforms defined in Lemma \ref{voronoi summation}. We state what we need in the form of a Lemma below.
\begin{lemma}\label{Bounds for integral transforms}
Let $a\geq 1$, $W: [1,2] \to \mathbb{C}$ be a compactly supported smooth function satisfying 
$$
W^{(l)}(x) \ll_{l} a^l,
$$
for all $l \geq 0$.
Define $W^{\pm}: \mathbb{R}^+ \to \mathbb{C}$ by
\begin{equation}\label{definition of integral transform W1}
W^\pm(\xi) = \int_0^\infty{W(x)J_g^\pm(4 \pi {}\sqrt{\xi x})}dx,
\end{equation}
where $J_g$ has been defined in Lemma \ref{voronoi summation}.
Then for all $j,l \geq 0$ and $\xi >0$
\begin{equation}
\xi^l \frac{\partial^l}{\partial \xi^l} W^{\pm}(\xi) \ll_{\mu_g,j,l} a^l\max\{1,\xi^{-\theta_g}\}\frac{a^j}{\xi^{j/2}}, 
\end{equation}
where $\mu_g$ is the parameter of $g$ at infinity and $\theta_g$ is defined as
\begin{equation}
\theta_g = \begin{cases}
0, \text{ if g is holomorphic}\\
\mathfrak{Im}(r), \text{if g is a Maass form and } \lambda_g = 1/2 + ir.
\end{cases}.
\end{equation}
Unconditionally, we have $\theta_g \leq \frac{7}{64}$ \cite{kim2003functoriality}. If $g$ is not exceptional, then
\begin{equation}
\xi^l \frac{\partial^l}{\partial \xi^l} W^{\pm}(\xi) \ll_{\mu_g,j,l} a^l\frac{a^j}{\xi^{j/2}}.
\end{equation}
\end{lemma}
\begin{proof}
Making the substitution $u = 2\pi \sqrt{\xi x}$ in \eqref{definition of integral transform W1}, we get
\begin{equation}
W^\pm(\xi) = \frac{1}{8 \pi^2 \sqrt{\xi}}\int_0^\infty{\frac{u}{\sqrt{\xi} }W\left(\frac{u^2}{(4 \pi)^2 \xi} \right)J_g^\pm(u)}du.
\end{equation}
Differentiating within the integral, we get
\begin{equation}
\left(\frac{\partial}{\partial \xi}\right)^l W^\pm(\xi) = \sum_{0 \leq k \leq l}{\left(\frac{\partial}{\partial \xi}\right)^k \left(\frac{1}{8 \pi^2 \sqrt{\xi} } \right)\int_0^\infty{ \left(\frac{\partial}{\partial \xi}\right)^{l-k} \left(\frac{u}{\sqrt{\xi} }W\left(\frac{u^2}{(4 \pi)^2 \xi} \right) \right)J_g^\pm(u)}du}.
\end{equation}
We apply Lemma 6.1 \cite{kowalski2002rankin} (see remark below)  with $$h(u) = \left(\frac{\xi }{a}\right)^{l-k}\left(\frac{\partial}{\partial \xi}\right)^{l-k} \left(\frac{u}{\sqrt{\xi} }W\left(\frac{u^2}{(4 \pi)^2 \xi} \right) \right)$$ and $M = 4\pi \sqrt{\xi}$ to bound the integral. Using our definition of $J_g^\pm$ (see Lemma \ref{voronoi summation}), we see that the real part of the Bessel function satisfies $\mathfrak{Re}\nu \geq -2\theta_g$. Thus if $\mathfrak{Re}\nu \leq 0$, then $$M^{\mathfrak{Re}\nu} \ll \max\{1,\xi^{-\theta} \}.$$
\end{proof}

\begin{remark}
We would like to point out a typo in Lemma 6.1 of \cite{kowalski2002rankin}. While the estimate stated in Lemma 6.1 \cite{kowalski2002rankin} is
\begin{equation}
\int_{0}^\infty{J_\nu(x)h(x)dx} \ll_{v,j} \frac{a^j(1+ |\log M |)}{M^{j-1}}\frac{M^{\mathfrak{Re \nu} + j +1}}{(1+ M )^{\mathfrak{Re \nu} + j +1/2}},
\end{equation}
the estimate that has been shown is
\begin{equation}
\int_{0}^\infty{J_\nu(x)h(x)dx} \ll_{v,j} \frac{a^j(1+ |\log M |)}{M^{j-1}}\frac{M^{\mathfrak{Re \nu} + j}}{(1+ M )^{\mathfrak{Re \nu} + j +1/2}}.
\end{equation}
Notice the absence of $+1$ in the exponent $M^{\mathfrak{Re \nu} + j +1}$. Furthermore, the same inequality holds with the $J$-bessel function replaced by $K$-Bessel function or $Y$-Bessel function, without the $\frac{M^{\mathfrak{Re} \nu}}{(1+M)^{\mathfrak{Re} \nu}}$ term.

The result above has been stated for $W$ having compact support contained in the interval $[1,2]$. But it holds without any change for the support contained in any absolutely bounded interval, bounded away from zero. For example, $\text{supp}(W) \subset [1/1000,1000]$ is sufficient for the purpose of this paper.
\end{remark}

\begin{definition}[$H_\theta$]\label{H theta}
We say that Hecke-cusp form $f$ of level $q$ and nebentypus $\chi$ satisfies the Hypothesis $H_\theta$, if for all $n \geq 1$ and any $\epsilon >0$
$$
| \lambda_f(n)| \ll_{\epsilon} n^{\theta+\epsilon},
$$
where $\lambda_f(p)$ are the local parameters of $\pi_f$ at $p$ and $\lambda_f(n)$ satisfy Hecke relations.
\end{definition}
The Ramanujan conjecture asserts that we can take $\theta =0$, for all $q$ and $\chi$. We can unconditionally take $\theta=7/64$ for Maass forms (Kim-Sarnak \cite{kim2003functoriality}) and $\theta =0$ for holomorphic forms. Rankin-Selberg theory \cite[2.28]{harcos2006subconvexity}, implies the Ramanujan conjecture on average unconditionally.
\begin{lemma}
Let $g$ cuspidal automorphic form of level $p$ and nebentypus $\chi_f$. Then for all $N \geq 1$ and $\epsilon >0$
\begin{equation}\label{rankin selberg bound}
\sum_{1 \leq n \leq N}{|\lambda_f(n)|^2} \ll_{\epsilon}(p(1+|\mu_f|) N)^\epsilon N, 
\end{equation}
where $\mu_f$ is the local parameter of $f$ at infinity.
\end{lemma}
We prove a Polya-Vinogradov type inequality for smooth partial sums of $\lambda_f(n) \lambda_g(n)$ below:
\begin{lemma}\label{Polya-Vinogradov type inequality}
Let $f$ and $g$ be cuspidal newforms of levels $p$ and $q$ respectively ,$W: \mathbb{R}^+ \to \mathbb{C}$ be a smooth compactly supported function and $Q(f \times g)$ be the arithmetic conductor of $L(f \times g,s)$ \eqref{arithmetic conductor}. If $f \neq \overline{g}$, so that $L(f \times g,s)$ has no poles, then
$$
\sum_{n}{\lambda_f(n)\lambda_g(n)W\left(\frac{n}{N} \right)} \ll_{\epsilon} (pqN)^\epsilon \sqrt{N \sqrt{Q(f \times g)}}.
$$

\end{lemma}

\begin{proof}

Let $F(s) = \sum_{n =1}^\infty{\frac{\lambda_f(n)\lambda_g(n)}{n^s}} =\frac{L(f \times g,s)}{L(\chi_f\chi_g,2s)} $. 
Using Mellin inversion,
\begin{align}\label{mellin inversion}
\sum_{n}{\lambda_f(n)\lambda_g(n)W\left(\frac{n}{N} \right)} = \frac{1}{2\pi i}\int_{(2)}{N^sF(s) \widetilde{W}(s)ds},
\end{align}
where 
\begin{equation}
\widetilde{W}(s) = \int_0^\infty{W(x)x^{s-1} dx}
\end{equation}
is the Mellin transform. Using integration by parts repeatedly, we can show that for all $A \geq 0$,
$$
\widetilde{W}(s) \ll_A \min\{ 1, |s|^{-A} \}.
$$
For $\mathfrak{Re}(s) > \frac{1}{2}$, $L(\chi_f\chi_g,2s)$ doesn't vanish. Hence, $F(s)$ is analytic in the same region. Let $\chi$ be any Dirichlet character. The following lower bound for $L(\chi,s)$ to the right of critical strip is elementary and well known.
\begin{align*}
L(\chi,1+\epsilon +it) &= \prod_{l \text{ prime} }{\left(1- \frac{\chi(l)}{l^{1+\epsilon +it}}\right)^{-1}}\\
&\geq \prod_{l \text{ prime} }{\left(1 + \frac{1}{l^{1+\epsilon}}\right)^{-1}}\\
&\geq \frac{c}{\zeta(1+\epsilon)} \gg \epsilon,
\end{align*}
where the constant $c$ is independent of $\chi$, $t$, and $\epsilon$. Shifting the contour to $\mathfrak{Re}(s) =1/2 + \epsilon$ in \eqref{mellin inversion}, we have
\begin{align}
\sum_{n}{\lambda_f(n)\lambda_g(n)W\left(\frac{n}{N} \right)} &= \frac{1}{2\pi i}\int_{(1/2+ \epsilon)}{N^sF(s) \widetilde{W}(s)ds}\\
&= \frac{1}{2\pi i}\int_{\substack{(1/2+ \epsilon) \\ |\mathfrak{Im}(s)| \leq (pqN)^\epsilon } }{N^s\frac{L(f \times g,s)}{L(\chi_f\chi_g,2s)} \widetilde{W}(s)ds} + O((pqN)^{-2018})\\
& \ll_{\epsilon} (pqN)^\epsilon N^{1/2} {Q(f \times g)}^{1/4}.
\end{align}

In the last line, we have used the Phragmen-Lindelof convexity bound for $L(f \times g,s)$ on $\mathfrak{Re}(s) =\frac{1}{2}+\epsilon$ to bound $L(f \times g,s)$ from above.
\end{proof}
As we shall make repeated use of the bound proved below, we state it in the form of a lemma.
\begin{lemma}\label{Bounds for convolution of fourier coefficients on AP}
Let $a,b,c \in \mathbb{N}$ be such that $(ab,c) =1$ and $X,Y \geq 1$. If $0 \leq \alpha, \beta \leq \frac{1}{2}$, then
\begin{equation}\label{bounds on fourier coefficient alpha,beta}
S(\alpha,\beta)=\sum_{\substack{n \leq X\\m \leq Y\\am \equiv bn(c)}}{\frac{|\lambda_f(n)|}{n^\alpha}\frac{|\lambda_g(m)|}{m^\beta} } \ll_{\epsilon} (pqXY)^{\epsilon}\frac{X^{1-\alpha}Y^{1-\beta}}{c}\left(1 + \sqrt{\frac{c}{Y}}+\sqrt{\frac{c}{X}} + \frac{c}{\sqrt{XY}}\right).
\end{equation}
In particular,
\begin{equation}\label{Equation for bounds of fourier coefficinet in AP}
S := S(0,0) =\sum_{\substack{n \leq X\\m \leq Y\\am \equiv bn(c)}}{|\lambda_f(n)||\lambda_g(m)|} \ll_{\epsilon} (pqXY)^{\epsilon}\frac{XY}{c}\left(1 + \sqrt{\frac{c}{Y}}+\sqrt{\frac{c}{X}} + \frac{c}{\sqrt{XY}}\right).
\end{equation}
\end{lemma}
\begin{proof}
For $\alpha = \beta =0$, using Cauchy-Schwarz inequality, we have
\begin{align}
S = S(0,0) &\ll \left(\sum_{\substack{n \leq X\\m \leq Y\\am \equiv bn(c)}}{|\lambda_f(n)|^2 }\right)^{1/2}\left(\sum_{\substack{n \leq X\\m \leq Y\\am \equiv bn(c)}}{|\lambda_g(m)|^2}\right)^{1/2}.
\end{align} 
Now applying the Rankin-Selberg bound \eqref{rankin selberg bound}, we get
\begin{equation}
S \ll (pqXY)^\epsilon \left(X\left(\frac{Y}{c}+1\right) \right)^{1/2}\left(Y\left(\frac{X}{c}+1 \right) \right)^{1/2}.
\end{equation}
For general $\alpha \leq 1/2$ and $\beta \leq 1/2$, the bound \eqref{bounds on fourier coefficient alpha,beta} follows from \eqref{Equation for bounds of fourier coefficinet in AP}, after performing a dyadic sub-division of the sum over $m,n$.

\end{proof}
\begin{remark}
If both $X,Y$ are greater than $c$, then the first term on right hand side of \eqref{Equation for bounds of fourier coefficinet in AP} dominates the others. This bound is optimal upto $(XY)^\epsilon$.
\end{remark}

\subsection{Approximate functional equation}

We refer the reader to \cite[Section 3]{harcos2006subconvexity} for proofs. For $s$ on the critical line, we set
\begin{equation}
A = \prod_{i=1}^4{|s + \mu_{f \times g,i}|^{1/2} },
\end{equation}
where the local parameters $\mu_{f \times g,i}$ of $\pi_f \times \pi_g$ can be computed in terms of the local parameters of $\pi_f$ and $\pi_g$ respectively. We can check that
$$
A \leq (|s| + \mu_f + \mu_g)^2.
$$
We have essentially isolated the spectral part of the analytic conductor as $A$. Let us define
\begin{equation}\label{S f times g}
S_{f \times g}(N)= \sum_{n}{\lambda_f(n)\lambda_g(n)U\left(\frac{n}{N}\right)},
\end{equation}
where $U: \mathbb{R} \to \mathbb{R}$ is a smooth function with compact support contained in $[1/2,5/2]$. 
By standard techniques, for $\mathfrak{Re}s =1/2$ and any $K \geq 1$ we can show
\begin{equation}
L(f \times g,s) \ll_A \log^2(pqA +1)\sum_{N}{\frac{|S_{f \times g}(N)|}{\sqrt{N}}\left(1+ \frac{N}{AQ(f \times g)}\right)^{-K}},
\end{equation}
where $N$ runs over reals of the form $N=2^\nu$, $\nu \geq -1$.
Thus to prove Theorem \ref{main theorem}, it is enough to prove the following statement.
\begin{proposition}\label{main proposition}
For $\delta= \frac{1}{64}$, any $0 \leq \epsilon \leq 10^{-6}$, and any $1 \leq N \leq (Apq)^{1+\epsilon}$,
\begin{equation}\label{Aim}
S_{f \times g}(N) \ll_{A,\epsilon} \frac{\sqrt{N}(pq)^{1/2+\epsilon}}{\max \{p,q \}^{\delta}},
\end{equation}
where the dependence on $A$ is polynomial.

\end{proposition}

If $N \leq \frac{pq}{\max \{p,q \}^{2\delta}}$, then \eqref{Aim} follows from Rankin-Selberg bound \eqref{rankin selberg bound}. Hence, we may assume 
\begin{equation} \label{bounds on N}
\frac{pq }{\max\{p,q \}^{2\delta} }\leq N \leq (Apq)^{1+\epsilon}.
\end{equation}
We shall define $T$ by 
\begin{equation}\label{definition of T}
T := \frac{pq}{N}.
\end{equation}
Then $(Apq)^{-\epsilon} \leq T \leq \max\{p,q \}^{2\delta}$.
Similarly to prove Theorem \ref{main theorem when f is fixed}, it is enough to prove the following statement:
\begin{proposition}\label{main proposition when f is fixed}
For $\delta= \frac{1-2\theta}{27+28\theta}$, any $0 \leq \epsilon \leq 10^{-6}$, and any $1 \leq N \leq (Apq)^{1+\epsilon}$,
\begin{equation}
S_{f \times g}(N) \ll_{A,p,\epsilon} \sqrt{N}q^{1/2-\delta+\epsilon},
\end{equation}
where the dependence on $A$ and $p$ is polynomial.
\end{proposition}
We may likewise use the Rankin-Selberg bound \eqref{rankin selberg bound} to show Proposition \ref{main proposition when f is fixed}, when $N \leq q^{1-2\delta}$. Proposition \ref{main proposition} and \ref{main proposition when f is fixed} are proved in Section \ref{proof of main proposition}. We shall treat both the propositions simultaneously till Section \ref{shifted convolution problem}.

\section{Initial steps: Amplification and Circle method}

Let us assume without loss of generality that $p < q$. We shall consider
\begin{equation}\label{Definitio of S(N)}
S(N)= S_{f \times g}(N) = \sum_{n \in \mathbb{Z}} {\lambda_f(n) \lambda_g(n)U\left(\frac{n}{N} \right)}.
\end{equation}
We begin by ``amplifying" the sum. We shall use the idea of Duke, Friedlander, and Iwaniec \cite{duke2002subconvexity} to amplify the sum using the $GL(2)$ Hecke-relations.
If $(l,q)=1$ and $l$ is prime, then
$$
\lambda_g(l)^2 - \lambda_g(l^2) = 1.
$$
Let 
\begin{equation}
\mathcal{L} = \{l \in [L/2,L]: l \text{ is prime }, (l,pq) =1 \},
\end{equation}
where $L (\leq q^{1/2})$ is a parameter to be chosen. $L$ should be thought of as a small power of $q$. Note that $| \mathcal{L} | \gg \frac{L}{\log L}$. Define the amplifier $\alpha$ by
\begin{equation}\label{definition of amplifier}
\alpha_r =
\begin{cases}
\lambda_g(l), \text{ if } r =l \in \mathcal{L},\\
-1, \text{ if } r = l^2 \text{ and } l \in \mathcal{L},\\
0 \text{ otherwise}.
\end{cases}
\end{equation}
In what follows, we shall be able to save at most $\sqrt{L}$ over the trivial bound for $S(N)$ (see the diagonal contribution \ref{Bounds on diagonal contribution}). As we need to prove Proposition \ref{main proposition}, i.e show that
$$
S_{f \times g}(N) \ll_{A,\epsilon} \frac{\sqrt{N}(pq)^{1/2+\epsilon}}{\max \{p,q \}^{\delta}},
$$
we can make the assumption that
$$
\max\{p,q\}^{\delta} \leq \sqrt{L},$$ where $\delta$ is as in Proposition \ref{main proposition} and Proposition \ref{main proposition when f is fixed}. This implies that (see \eqref{definition of T})
\begin{equation}\label{upper bound for T}
T \leq \max\{p,q \}^{2\delta} \leq L.
\end{equation}

\begin{lemma}
Let
\begin{equation}\label{definition of S1}
S_1(N) = \frac{1}{|\mathcal{L}|}\sum_{r \leq L^2}{\alpha_r}\sum_{n}{\lambda_f(n)\lambda_g(nr)U(n/N)}.
\end{equation}
Then, for all $L \geq 100$,
\begin{equation}\label{amplification}
S(N) = S_1(N)+ O((pqN)^\epsilon N/\sqrt{L}).
\end{equation}
\end{lemma}
\begin{proof}

Using multiplicativity of the Fourier coefficients $\lambda_g$, $S_1(N)$ can be rewritten as
{\small
\begin{align}
&\frac{1}{|\mathcal{L}|}\sum_{r}{\alpha_r \lambda_g(r)}\sum_{(n,l)=1}{\lambda_f(n)\lambda_g(n)U(n/N)}+\frac{1}{|\mathcal{L}|} \sum_{r}{\alpha_r}\sum_{l|n}{\lambda_f(n)\lambda_g(nr)U(n/N)},
\end{align}
}
where $l$ has been defined in \eqref{definition of amplifier}. Adding and subtracting the $n$ which are divisible by $l$ to the first term, we get
{\small
\begin{equation}\label{simplified amplifier step}
\begin{split}
S_1(N)=\frac{1}{|\mathcal{L}|}\sum_{r}{\alpha_r \lambda_g(r)}\sum_{n}{\lambda_f(n)\lambda_g(n)U(n/N)}&- \frac{1}{|\mathcal{L}|}\sum_{r}{\alpha_r \lambda_g(r)}\sum_{l|n}{\lambda_f(n)\lambda_g(n)U(n/N)} \\&+ \frac{1}{|\mathcal{L}|}\sum_{r}{\alpha_r}\sum_{l|n}{\lambda_f(n)\lambda_g(nr)U(n/N)}.
\end{split}
\end{equation}
}
The second two terms, by an application of Rankin-Selberg bound \eqref{rankin selberg bound} and $H_{1/4}$ (see \ref{H theta}), are seen to be bounded above by $O((pqN)^\epsilon NL^{-1/2})$. We have chosen the amplifier $\alpha$ such that
\begin{equation}\label{amplifier property}
\sum_r{\alpha_r \lambda_g(r) } =|\mathcal{L}|.
\end{equation}
Combining \eqref{simplified amplifier step} and \eqref{amplifier property}, we get 
$$
S(N) = S_1(N) + O\left((pqN)^\epsilon\frac{N}{\sqrt{L} }\right).
$$ 
\end{proof}

We separate the oscillation of $f$ and $g$ in \eqref{definition of S1} using circle method. The equality of integers $$nr=m,$$ can be rewritten as a congruence $$nr \equiv m \pmod M,$$ if the moduli $M$ is greater than $|nr-m|$. The main idea is to choose the moduli $M$ as multiples of the product of the levels of $f$ and $g$. 
Let $C = 10L^2$ be a parameter defined by $L$. (It is sufficient to choose $C = 10\frac{L^2}{T}$, but at the cost of more delicate analysis. We get a minor improvement in the result when this is done). Define the set of moduli $\mathcal{C}$ by
\begin{equation}\label{definition of moduli C}
\mathcal{C} = \{c \in [C,2C]: c \text{ is prime }, (c,pq) =1 \}.
\end{equation}
We note that for all $c \in \mathcal{C}$ and $l \in \mathcal{L}$, $c$ and $l$ are coprime. Let us define the weight function $W(x,y)$ to be the product $U(x)V(y)$, with $U$ being the weight function in \eqref{definition of S1} and $V:[1/3,2] \rightarrow \mathbb{R}$ being a smooth bump function which is identically $1$ on $[1/2,5/2]$. Since the support of $U$ is contained in $[1/2,5/2]$, if $U(x) \neq 0$, then $W(x,x) =1$. We would like to note the following bound on derivatives of $W$:
\begin{equation}\label{bound on derivatives of W}
\frac{\partial^j}{\partial y^j}\frac{\partial^i}{\partial x^i} W(x,y) \ll_{i,j} 1.
\end{equation}
Separating the oscillation of $\lambda_f$ and $\lambda_g$ by circle method in \eqref{definition of S1}, we get
\begin{align}
S_1(N) &= \frac{1}{|\mathcal{L}| }\sum_{r}{\alpha_r}\sum_{\substack{n , m } }{\lambda_f(n)\lambda_g(m)\delta \left(rn = m \right)W\left(\frac{n}{N},\frac{m}{rN}\right)} \\
&=\frac{1}{|\mathcal{L}| }\sum_{r}{\frac{1}{|\mathcal{C}|}}\sum_{c \in \mathcal{C}}{\alpha_r}\sum_{\substack{n , m } }{\lambda_f(n)\lambda_g(m)\delta \left(rn \equiv m \pmod{pqc}\right)W\left(\frac{n}{N},\frac{m}{rN}\right)} \\
&= \frac{1}{pq|\mathcal{L}||\mathcal{C}| }\sum_{r}\sum_{c \in \mathcal{C}}{\frac{1}{c}}\sum_{a (pqc)}{\alpha_r}\sum_{\substack{n, m  } }{\lambda_f(n)\lambda_g(m)e\left(\frac{a(nr-m)}{pqc} \right)W\left(\frac{n}{N},\frac{m}{rN}\right)}\\
&= \frac{1}{pq|\mathcal{L}| |\mathcal{C}|}\sum_{r}\sum_{c \in \mathcal{C}}{\frac{\alpha_r}{c}}\sum_{a (pqc)}\sum_{n }{\lambda_f(n)e\left(\frac{anr}{pqc} \right) }\sum_{m}{\lambda_g(m)e\left(\frac{-am}{pqc} \right)W\left(\frac{n}{N},\frac{m}{rN}\right)}. \label{circle method separated expression}
\end{align}
Thus
\begin{equation}\label{S1 in terms of S1r,c}
S_1(N)=\frac{1}{pq|\mathcal{L}| |\mathcal{C}|}\sum_{r}\sum_{c \in \mathcal{C}}{\frac{\alpha_r}{c}}{S_1^{r,c}(N)},
\end{equation}
where
\begin{equation}\label{definition of S1(r,c)}
S_1^{r,c}(N) = \sum_{a (pqc)}\sum_{n }{\lambda_f(n)e\left(\frac{anr}{pqc} \right) }\sum_{m}{\lambda_g(m)e\left(\frac{-am}{pqc} \right)W\left(\frac{n}{N},\frac{m}{rN}\right)}.
\end{equation}
The expression \eqref{definition of S1(r,c)} will be our starting point. We shall outline our argument in the next section.

\section{Outline of Proof \label{outline of proof}}

In this outline, we shall assume the Ramanujan conjecture. We indicate the proof in the case $N =pq$ as this is the most important case. We shall apply circle method to the following ``amplified" sum.
\begin{equation}\label{transition range}
S = \frac{1}{|\mathcal{L}|}\sum_{r \ll C}{\alpha_r}\sum_{n \sim pq} {\lambda_f(n) \lambda_g(nr)}.
\end{equation}
The notation $n \sim pq$ means that $n$ runs over natural numbers between $[pq,2pq]$, weighted by a smooth function. The symbol $A \rightsquigarrow B$, in this outline, means that $A$ transforms into $B$ after a series of steps.

Applying Voronoi summation in both the $m$ and $n$ variables to \eqref{circle method separated expression} yields 
\begin{align}
S \rightsquigarrow \frac{1}{pqC^3 |\mathcal{L} | }\sum_{r \leq C}\sum_{c \in \mathcal{C}}\sum_{a (pqc)}{\alpha_r}\sum_{\substack{n \sim pqC^2 \\ m \sim pqC } }{\lambda_f(n)\lambda_g(m)e\left(\frac{\overline{a}(-n\bar{r}+m)}{pqc} \right)}\\
\rightsquigarrow \frac{1}{C^2 |\mathcal{L} | }\sum_{|s| \ll C}\sum_{r \leq C}\sum_{c \in \mathcal{C}}{\alpha_r}\sum_{\substack{m \sim pqC  } }{\lambda_f(mr + pqcs)\lambda_g(m)} \label{expression for central value}.
\end{align}
We have pretended above that $(a,cpq)=1$ for all residue classes $a\pmod {cpq}$. While this is obviously false, this assumption captures the essence of the proof. We get Ramanujan sums instead of the complete sum over all additive frequencies, in the actual proof. This seems to be an annoying technical issue when writing down a complete proof (see Lemma \ref{Bound for q_1=q} and \ref{Bound for q_2=q}). Note that the trivial bound for the right hand side of \eqref{expression for central value} is $O((pqC)^{1+\epsilon})$ i.e we have gained a $C$ over the trivial bound. We separate the right hand side above into diagonal ($s=0$) + non-diagonal part. For the $s=0$ part: trivially bounding the $r$ and $c$ sum we get
$$
S(s=0) \ll \frac{1}{C}\left|\sum_{m \sim pqC}{\lambda_f(m)\lambda_g(m)}\right|.
$$
Since $pqC$ is greater than the square root of the conductor of $f \times g$, we can get a saving in this sum by using the functional equation.  Using Lemma \ref{Polya-Vinogradov type inequality} we get
$$
\sum_{m \sim pqC}{\lambda_f(m) \lambda_g(m)} \ll (pq)^{1+\epsilon}\sqrt{C}.
$$
Thus the diagonal part can be bounded by $O((pq)^{1+\epsilon}/\sqrt{C} )$. For the off-diagonal part i.e $s \neq 0$, we combine the variables $sc = t \ll C^2$ and use Cauchy-Schwarz to eliminate the oscillation due to $\lambda_g(m)$. This leads us to our shifted convolution problem for $f$. We get
\begin{equation}\label{SCP in circle method}
S^2 \lll \frac{pqc}{(C^2|\mathcal{L}| )^2}\sum_{r_1 ,r_2, t_1, t_2}{\alpha_{r_1}\overline{\alpha_{r_2}} }\sum_{m \sim pqC}{\lambda_f(r_1m +pqt_1)\overline{\lambda_f}(r_2m + pqt_2) }. 
\end{equation}
As we have squared the expression, we need to save $C^2$ on the right hand side. 
In the diagonal-terms $r_1t_2 =r_2t_1$, we save $|\mathcal{L}|C^2$ which is greater than $C^2$. For the off diagonal terms we note that the shift $pq(r_1t_2-r_2t_1)$ (defined in \eqref{definition of SCP}) is a multiple of $p$. Munshi \cite{munshi2017subconvexity} encounters a similar problem of bounding
$$
\mathcal{S}(X,h) = \sum_{n \sim pX}{\lambda_f(n)\lambda_f(n+ph)}.
$$
He shows a power saving for $\mathcal{S}(X,h)$, as long as $X \geq p^{\delta}$ for some $\delta >0$ (see \eqref{bound for SCP shift multiple of level}). In our scenario $X =q \geq p$. We briefly sketch an outline of this argument here. We can rewrite $\mathcal{S}$ as
\begin{align}
\mathcal{S} &= \sum_{n,m \sim pX}{\lambda_f(n)\lambda_f(m)\delta(m=n+ph)}\\
&= \sum_{\substack{n,m \sim pX \\ n + ph \equiv m (p)} }{\lambda_f(n)\lambda_f(m)\delta\left(\frac{n-m}{p} +h=0 \right)}
\end{align}

We pick the congruence mod $p$ using additive characters and use the Duke, Friedlander, and Iwaniec circle method \eqref{delta method} to rewrite the $\delta$ symbol. This leads us to a sum of the form
\begin{equation}
\mathcal{S} \rightsquigarrow \frac{1}{pX}\sum_{n,m \sim pX}\sum_{q \sim \sqrt{X}}\sum_{b(pq)*}{\lambda_f(n)\lambda_f(m)e\left(\frac{a(n-m-ph)}{pq} \right)}.
\end{equation}
Applying Voronoi summation \eqref{voronoi summation} to both $n$ and $m$ sum, we are led to
\begin{equation}
\mathcal{S} \rightsquigarrow \frac{1}{p}\sum_{n,m \sim p}\sum_{q \sim \sqrt{X}}{\lambda_f(n)\lambda_f(m)S(m-n,ph,pq)}.
\end{equation}
The important point to note here is that the Kloosterman sum modulo $pq$, factors as a Ramanujan sum modulo $p$ times a Kloosterman sum modulo $q$. The Ramanujan sum, being very small, allows us to gain the additional saving. Bounding the right hand side using the Weil bound for Kloosterman sum and $H_0$ (see \ref{H theta}), we get
\begin{equation}
\mathcal{S} \ll (pX^{3/4})^{1+\epsilon} = \frac{(pX)}{X^{1/4}}(pX)^\epsilon.
\end{equation}
Using spectral theory we should be able to improve this bound to $$\mathcal{S} \ll (pX)^\epsilon p\sqrt{X}.$$ We get a satisfactory bound for our problem in Section \ref{shifted convolution problem}, imitating Munshi's ideas. 

Although our version of circle method \eqref{circle method separated expression} looks trivial, the set of moduli we chose to capture the congruence in the circle method have inbuilt into them the levels of $f$ and $g$. This feature can be noticed in the work of Aggarwal, Holowinsky, Lin, and Sun in their simplification of Munshi's proof of Burgess bound.
We would like to point out here that one could instead solve this problem by considering the following amplified second moment
\begin{equation}\label{second moment}
\sum_{h \in \mathcal{A}(pq)}{|M_g(h)|^2|L(f \times h,s)|^2},
\end{equation}
where $\mathcal{A}(pq)$ runs over a weighted Hecke basis for automorphic forms of level $pq$ and 
$$
M_g(h) = \sum_{r}{\alpha_r\lambda_h(r)},
$$ is the Duke, Friedlander, and Iwaniec amplifier. Using the Kuznetsov trace formula to rewrite the spectral sum, we roughly
get
\begin{equation}
\sum_{r_1,r_2}{\alpha_{r_1}\alpha_{r_2}}\sum_{c \geq 1}{\frac{1}{cpq}}\sum_{n,m \sim pq}{\lambda_f(n)\overline{\lambda_f(m)} S(r_2n,r_1m,cpq)\varphi\left(\frac{4\pi \sqrt{nm}}{cpq} \right)}.
\end{equation}
 If we apply a Voronoi transformation to either the $n$ or $m$ sum, we end up with a shifted convolution problem of the following shape:
\begin{equation}\label{SCP in moment method}
\sum_{r_1,r_2,s}{\dots}\sum_{m}{\lambda_f(m)\overline{\lambda}_f(r_1r_2m + pqs)}.
\end{equation}
Writing the proof using circle method cleans up the proof considerably. But it is worth noting here that behind the scenes, we are implicitly computing the spectral second moment in \eqref{second moment}.

In \cite{kowalski2002rankin}, the authors compute such an amplified second moment \eqref{second moment}, when $p$ is fixed. But they use $\mathcal{A}(q)$ instead of $\mathcal{A}(pq)$. Let us consider the Petersson trace formula:
$$
\sum_{f \in  H_k(q)}{\omega_f^{-1} \lambda_f(n) \overline{\lambda_f}(m)}= \delta(m,n) +\sum_{c \geq 1}{\frac{S(m,n,c)}{cq} J_{k-1}\left(\frac{4 \pi \sqrt{mn}}{cq} \right)}.
$$

We are forced to consider all values of $c$ on the right hand side. They treat the large values of $c$ using a large sieve type inequality. The smaller $c$'s are treated as explained above. In the optimal treatment using this method we are forced to consider $c$ as large as $q^{1/6}$ (see below Equation 7.13 in \cite{kowalski2002rankin}). Using circle method we are easily able to restrict our attention to only the small $c$'s ($c \leq L^2$). This is the principal reason for our improvement in the exponent in Theorem \eqref{main theorem when f is fixed}. This is a technical problem which can be overcome by using a clever test function in the Kuznetsov trace formula (see Section 3.2, \cite{fouvry2015algebraic}). When this is done, the trace formula yields a better exponent. The limit of this method, under Ramanujan conjecture, is saving $q^{\frac{1}{20}}$ over the trivial bound for $L(f \times g,s )$.

\section{Voronoi transformations \label{voronoi transformations and simplifications}}

We shall apply Voronoi transformations \eqref{voronoi summation} to the $n$ and $m$ sums in \eqref{definition of S1(r,c)}. The modulus in Voronoi summation clearly depends on the greatest common divisor $(ar,cpq)$. By our choice $c,p,q$, and $r$ are pairwise co-prime. Let us assume that $(a,cpq) =c_1p_1q_1$, $c=c_1c_2'$, $p=p_1p_2'$, and $q=q_1q_2'$. Since we assumed $pq$ is square-free, using the Chinese remainder theorem, we rewrite \eqref{definition of S1(r,c)} as
{\small
\begin{equation}\label{expanded version of S1 r,c}
\begin{split}
S_1^{r,c}(N) = \sum_{\substack{p_1p_2' =p\\q_1q_2'=q}}\sum_{c_1c_2'=c }\sum_{a (p_2'q_2'c_2')^*}&\left(\sum_{n }{\lambda_f(n)e\left(\frac{anr}{p_2'q_2'c_2'} \right) } \right)\\
&\left(\sum_{m}{\lambda_g(m)e\left(\frac{-am}{p_2'q_2'c_2'} \right)} \right)W\left(\frac{n}{N},\frac{m}{rN}\right) .
\end{split}
\end{equation}
}
Due to our assumption that $q$ is prime, $q_1$ is either $1$ or $q$ (similarly for $p$ and $c$). In Lemma \ref{Bound for q_1=q}, we show that the contribution of $q_1=q$ and $c_1 =c$ to $S_1(N)$ are negligible. We can make a similar statement about the contribution of $p_1 =p$, if the value of $p$ is large. On a first reading, we encourage the reader to assume $p_1=1$, $q_1=1$, and $c_1=1$, in order to avoid unnecessary complications in the notation. Applying Voronoi summation \eqref{voronoi summation} to the $n$ and $m$ sum modulo $p_2'q_2'c_2'$, we get
\begin{equation}\label{S1 in dual side}
S_1^{r,c}(N) = \sum_{\pm,\pm}{S_1^{r,c,(\pm,\pm)}(N)},
\end{equation}
where 
{\small
\begin{equation}\label{Voronoi transformed S1 r,c}
\begin{split}
S_1^{r,c,(\pm,\pm)}(N) = \sum_{\substack{p_1p_2' =p\\q_1q_2'=q}}\sum_{c_1c_2'=c }\sum_{a (p_2'q_2'c_2')^*}&\left(\sum_{n }{\overline{\lambda_f}(n)e\left(\frac{\mp\overline{ap_1r}n}{p_2'q_2'c_2'} \right) } \right)\\
&\left(\sum_{m}{\overline{\lambda_g}(m)e\left(\frac{\pm\overline{aq_1}m}{p_2'q_2'c_2'} \right)}\right)
W^{\pm,\pm}\left(\frac{n}{X},\frac{m}{Y}\right),
\end{split}
\end{equation}
}
\begin{equation}
W^{\pm,\pm}\left(\xi,\eta\right)= \frac{rN^2}{(p_2'q_2'c_2')^2\sqrt{p_1q_1}}\int_0^\infty\int_0^\infty{W(x,y)J^{\pm}_f\left(4\pi \sqrt{\xi x}\right) J^{\pm}_g\left(4 \pi \sqrt{\eta y}\right)dxdy},
\end{equation}
and
\begin{equation}\label{defintion of X and Y}
X = \frac{p_1(p_2'q_2'c_2')^2}{N}, Y = \frac{q_1(p_2'q_2'c_2')^2}{rN}.
\end{equation}
Using Lemma \ref{Bounds for integral transforms} and the assumption that $f$ and $g$ are not exceptional, we see that for all $j \geq 0$,
\begin{equation}\label{Bound on W1}
W^{\pm,\pm}\left(\frac{n}{X},\frac{m}{Y}\right) \ll_j \frac{rN^2}{(p_2'q_2'c_2')^2\sqrt{p_1q_1}}\min\left\{1, \left(\frac{X}{n}\right)^{-j},\left(\frac{Y}{m}\right)^{-j},\left( \frac{XY}{mn}\right)^{-j} \right\}.
\end{equation}
Among the four possible choices in $\{\pm,\pm\}$, we shall restrict our attention to $\{+,+\}$, as this is prototypical. Let us denote $S_1^{r,c,(+,+)}(N)$ by $S_2^{r,c}(N)$. We denote
\begin{equation}\label{S1 dual in terms of S2r,c}
S_1^{dual}(N) :=\frac{1}{pq|\mathcal{L}| |\mathcal{C}|}\sum_{r}\sum_{c \in \mathcal{C}}{\frac{\alpha_r}{c}}{S_2^{r,c}(N)}.
\end{equation}
By repeating the proof in a very straightforward way for all the other choices of $\{\pm,\pm\}$ i.e $\{+,-\}, \{-,+ \}, $ and $\{-,- \}$, we can check that the bound obtained for $S_1^{dual}(N)$ holds for $S_1(N)$ also. We dyadically divide the sum over $n$ and $m$ in \eqref{Voronoi transformed S1 r,c}, using a smooth partition of unity. As this is a standard technique, we refer to Lemma 1 of \cite{munshi2015annals} for details. This gives
{\small
\begin{equation}
\begin{split}
S_2^{r,c}(N) = \sum_{\substack{(A,\rho_1)\\(B,\rho_2)} }&\sum_{\substack{p_1p_2' =p\\q_1q_2'=q}}\sum_{c_1c_2'=c }\sum_{a (p_2'q_2'c_2')^*}\left(\sum_{n }{\overline{\lambda_f}(n)e\left(\frac{\mp\overline{ap_1r}n}{p_2'q_2'c_2'} \right) } \right)\\
&\left(\sum_{m}{\overline{\lambda_g}(m)e\left(\frac{\pm\overline{aq_1}m}{p_2'q_2'c_2'} \right)}\right)
W^{+,+}\left(\frac{n}{X},\frac{m}{Y}\right)\rho_1\left(\frac{n}{A}\right)\rho_2\left(\frac{m}{B}\right),
\end{split}
\end{equation}
}
where the pairs $(\rho_1,A)$ and $(\rho_2,B)$ are locally finite smooth partitions of unity (see \cite{munshi2015annals} for more details). The important point is that the support of $\rho_j$ is contained in $[1/2,1]$, for all $\rho_j$ appearing in the partition of unity. Hence the sum over $n$ and $m$ runs smoothly over integers in $[A/2,A]$ and $[B/2,B]$ respectively. It is also convenient to re-normalize the weight function $W^{+,+}$, so that it is absolutely bounded. To this end, we define
\begin{equation}\label{definition of W2}
W_2(\xi,\eta) = \frac{(p_2'q_2'c_2')^2\sqrt{p_1q_1}}{rN^2}W^{+,+}\left(\frac{A\xi}{X},\frac{B\eta}{Y}\right)\rho_1(\xi)\rho_2(\eta).
\end{equation}
Then, $W_2: \mathbb{R}^2 \to \mathbb{C}$ is a smooth function with compact support contained in $[1/2,1] \times [1/2,1]$. Lemma \ref{Bounds for integral transforms} implies that for all $i,j \geq 0$
\begin{equation}\label{bound on derivatives of W2}
\frac{\partial^j}{\partial \eta^j}\frac{\partial^i}{\partial \xi^i}W_2(\xi,\eta) \ll_{i,j} 1.
\end{equation}
If $n \geq (pq)^\epsilon X$ or $m \geq (pq)^\epsilon Y$, choosing $j =2018 \epsilon^{-1}$ in \eqref{Bound on W1}, we get $W^{\pm,\pm} \ll_{\epsilon} \frac{1}{(pq)^{1000} }$. Hence the contribution of such terms to $S_1(N)$ is negligible.  Thus
\begin{equation}\label{bounding S1 in terms of S2}
S_2^{r,c}(N) = \sum_{\substack{p_1p_2' =p\\q_1q_2'=q}}\sum_{c_1c_2'=c } \sum_{\substack{(\rho_1, A \leq (pq)^\epsilon X) \\ (\rho_2, B \leq (pq)^\epsilon Y) } }{S_2^{r,c}(A,B) },
\end{equation}
where $A$ and $B$ runs over powers of $2$, $\rho_1$ and $\rho_2$ are the smooth functions arising from the partitions of unity, and
{\small
\begin{equation}
\begin{split}
S_2^{r,c}(A,B) = &{ \frac{rN^2}{(p_2'q_2'c_2')^2\sqrt{p_1q_1}}}\sum_{a (p_2'q_2'c_2')^*}\\
&\left(\sum_{n }{\overline{\lambda_f}(n)e\left(\frac{-\overline{ap_1r}n}{p_2'q_2'c_2'} \right) } \right)\left(\sum_{m}{\overline{\lambda_g}(m)e\left(\frac{\overline{aq_1}m}{p_2'q_2'c_2'} \right)}\right)
W_2\left(\frac{n}{A},\frac{m}{B}\right).
\end{split}
\end{equation}
}
Summing up over the reduced classes $a$ modulo $p_2'q_2'c_2'$, we get Ramanujan sums. This gives
{\small
\begin{equation}\label{S2 r,c}
\begin{split}
S_2^{r,c}(A,B) = {\frac{rN^2}{(p_2'q_2'c_2')^2\sqrt{p_1q_1}}}\left(\sum_{\substack{n,m }}{\overline{\lambda_f}(n)\overline{\lambda_g}(m)\mathfrak{r}_{p_2'q_2'c_2'}(p_1rm - q_1n)} W_2\left(\frac{n}{A},\frac{m}{B}\right)\right),
\end{split}
\end{equation}
}
where 
\begin{equation}\label{Ramanujan sum}
\mathfrak{r}_q(n) = \sum_{a(q)*}{e\left(\frac{aq}{n}\right)} = \sum_{d|(q,n)}{\mu\left(\frac{q}{d}\right)d},
\end{equation}
is the Ramanujan sum. We shall use the identity $\eqref{Ramanujan sum}$ to expand the Ramanujan sum. This gives
\begin{equation}\label{S2 in terms of S3}
S_2^{r,c}(N) = \sum_{\substack{p_1p_2p_3 =p\\q_1q_2q_3=q}}\sum_{c_1c_2c_3=c } \sum_{\substack{(\rho_1, A \leq (pq)^\epsilon X) \\ (\rho_2, B \leq (pq)^\epsilon Y) } }{S_3^{r,c}(A,B) },
\end{equation}
where
{\small
\begin{equation}\label{S_3 r,c}
\begin{split}
S_3^{r,c}(A,B) = &{\mu(p_2)\mu(q_2)\mu(c_2) \frac{rN^2}{(p_2q_2c_2)^2p_3q_3c_3\sqrt{p_1q_1}}}\times\\
&\left(\sum_{\substack{n,m}}{\overline{\lambda_f}(n)\overline{\lambda_g}(m) \delta[p_1rm \equiv q_1n (p_3q_3c_3) ]} W_2\left(\frac{n}{A},\frac{m}{B}\right)\right),
\end{split}
\end{equation}
}
\begin{equation}\label{defintion of X and Y}
X = \frac{p_1(p_2p_3q_2q_3c_2c_3)^2}{N}, Y = \frac{q_1(p_2p_3q_2q_3c_2c_3)^2}{rN}.
\end{equation}

We have made the substitution $p_2'= p_2p_3$, $q_2' = q_2q_3$, and $c_2'=c_2c_3$. Since $q$ is prime, $q_2$ is either $1$ or $q$ (similarly for $p$ and $c$). Lemma \ref{Bound for q_2=q} shows that the contribution of $q_2=q$ and $c_2= c$ to $S_2(A,B)$ is negligible (The contribution of $p_2=p$ is also be shown to be negligible, if $p$ is large). On a first reading, we encourage the readers to assume $p_3=p$, $q_3=q$, and $c_3=c$ . 

Before we proceed further we would like to get rid of the boundary cases which make a smaller contribution. If we do not assume the primality of $p$ and $q$, we would have to consider the various factorizations possible and give a separate argument for each of these, depending on the sizes of the factors. Though straightforward, it is messy.

\begin{lemma}\label{Bound for q_1=q}
Let $S_2^{r,c,q_1 =q}(N)$ be the contribution of the terms with $q_1=q$ to \eqref{S2 in terms of S3}. Let $S_1^{dual,q_1 =q}$ be the contribution of such terms to $S_1^{dual}(N)$ \eqref{S1 dual in terms of S2r,c} i.e
\begin{equation}\label{S_1^q_1=q}
S_1^{dual,q_1 =q}(N) = \frac{1}{pq|\mathcal{L}| |\mathcal{C}|}\sum_{r}\sum_{c \in \mathcal{C}}{\frac{\alpha_r}{c}}{S_2^{r,c,q_1=q}(N)}.
\end{equation}
Then
\begin{equation}\label{Bound for S1 q_1=q}
S_1^{dual,q_1 =q}(N) \ll (pq)^\epsilon \sqrt{Npq} \left(\frac{C\sqrt{L}}{q^{3/2}}+\frac{\sqrt{C}}{q}\right).
\end{equation}
Similarly, let $S_1^{dual,p_1 =p}(N)$ and $S_1^{dual,c_1 =c}(N)$ be the contribution of the terms with $p_1=p$ and $c_1=c$ respectively to \eqref{S1 dual in terms of S2r,c}. Then
\begin{equation}\label{Bound for S1 p_1=p}
S_1^{dual,p_1 =p}(N) \ll (pq)^\epsilon \sqrt{Npq} \left(\frac{C\sqrt{L}}{p^{3/2}}+\frac{\sqrt{C}}{p}\right)
\end{equation}
and
\begin{equation}\label{Bound for S1 c_1=c}
S_1^{dual,c_1 =c}(N) \ll (pq)^\epsilon \frac{\sqrt{N pq} L}{C} .
\end{equation}

\end{lemma}

\begin{proof}
If $q_1 =q$, then $q_2 =q_3 =1$. Using Lemma \ref{Bounds for convolution of fourier coefficients on AP} to bound the right hand side of \eqref{S_3 r,c}, we have
\begin{equation}
S_3^{r,c}(A,B) \ll \frac{rN^2}{(p_2c_2)^2p_3c_3\sqrt{p_1q}}\frac{AB}{p_3c_3}\left(1 + \sqrt{\frac{p_3c_3}{A}} + \sqrt{\frac{p_3c_3}{B}} + \sqrt{\frac{(p_3c_3)^2}{AB}} \right).
\end{equation}
The function on right hand side of the inequality is increasing in $A$ and $B$. Thus $S_3(A,B)$ is bounded by $S_3((pq)^\epsilon X,(pq)^\epsilon Y)$, for all $A \leq (pq)^\epsilon X$ and $B \leq (pq)^\epsilon Y$. Moreover, among all factorizations of $p_1p_2p_3 =p$ and $c_1c_2c_3= c$, the maximum value is attained at $p_3=p$ and $c_3 =c$. In this case, $X= \frac{(pc)^2}{N}$ and $Y = \frac{q(pc)^2}{rN}$ \eqref{defintion of X and Y}. Putting this in \eqref{S2 in terms of S3}, we get
\begin{align}
S_2^{r,c,q_1=q} &\ll (pq)^\epsilon\frac{rN^2}{(pc)^2\sqrt{q}}\frac{(pc)^2}{N}\frac{q(pc)^2}{rN}\left(1+ \sqrt{\frac{N}{pc}} + \sqrt{\frac{rN}{pqc}} + \sqrt{\frac{rN^2}{q(pc)^2} }\right)\\
&\ll (pq)^\epsilon \left(\frac{(pc)^2 q}{\sqrt{cT}} + (pc)^2\sqrt{q}\right).
\end{align}
Finally, using this bound in \eqref{S_1^q_1=q},
\begin{align}
S_1^{dual,q_1=q}(N) &\ll (pq)^\epsilon\frac{1}{pq|\mathcal{L}| |\mathcal{C}|}\sum_{r}\sum_{c \in \mathcal{C}}{\frac{|\alpha_r|}{c}}{\left( (pc)^2\sqrt{q}+ \frac{(pc)^2 q}{\sqrt{cT}} \right)}\\
&\ll (pq)^\epsilon \left(\frac{pqC}{q^{3/2}}+\frac{\sqrt{NpqC}}{q}\right)\\
&\ll (pq)^\epsilon \sqrt{Npq} \left(\frac{C\sqrt{L}}{q^{3/2}}+\frac{\sqrt{C}}{q}\right)
\end{align}
We can establish the bound for $S_1^{dual,p_1=p}$ similarly. If we follow the same method for $S_1^{dual,c_1 =c}(N)$, we obtain
\begin{equation}
S_1^{dual,c_1=c}(N) \ll (pq)^\epsilon \frac{\sqrt{Npq} L}{C}.
\end{equation}
We can prove slightly better bounds for these sums, but it is unnecessary for our purpose.
\end{proof}

We exhibit a satisfactory bound for the contribution of terms with $q_2 =q$ to \eqref{S2 in terms of S3} in the next lemma.
\begin{lemma}\label{Bound for q_2=q}
Let $S_2^{r,c,q_2=q}$ be the contribution of terms with $q_2 =q$ to $S_2^{r,c}(N)$ \eqref{S2 in terms of S3}. Let $S_1^{dual,q_2 =q}$ be the contribution of such terms to $S_1^{dual}(N)$ \eqref{S1 dual in terms of S2r,c} i.e
\begin{equation}\label{S_1^q_2=q}
S_1^{dual,q_2 =q}(N) = \frac{1}{pq|\mathcal{L}| |\mathcal{C}|}\sum_{r}\sum_{c \in \mathcal{C}}{\frac{\alpha_r}{c}}{S_2^{r,c,q_2=q}(N)}.
\end{equation} Then
\begin{equation}\label{Bound on S_2^q_2=q }
S_1^{dual,q_2=q}(N) \ll (pq)^\epsilon \frac{NL}{q}.
\end{equation}
Similarly, let $S_1^{dual,p_2=p}$  and $S_1^{dual,c_2=c}$be the contribution of terms with $p_2 =p$ and $c_2=c$ to $S_1^{dual}(N)$ respectively. Then
\begin{equation}\label{Bound on S_2^p_2=p}
S_1^{dual,p_2=p}(N) \ll (pq)^\epsilon \frac{NL}{p} .
\end{equation}

\begin{equation}\label{Bound on S_2^c_2=c}
S_1^{dual,c_2=c}(N) \ll (pq)^\epsilon \frac{NL}{C} .
\end{equation}
\end{lemma}
\begin{proof}
If $q_2=q$, then $q_1=q_3=1$. Then
{\small
\begin{equation}\label{S_3 r,c q_2=q}
\begin{split}
S_3^{r,c,q_2=q}(A,B) = &{\mu(p_2)\mu(c_2) \frac{rN^2}{(p_2qc_2)^2p_3c_3\sqrt{p_1}}}\times\\
&\left(\sum_{\substack{n,m}}{\overline{\lambda_f}(n)\overline{\lambda_g}(m) \delta[p_1rm \equiv n (p_3c_3) ]} W_2\left(\frac{n}{A},\frac{m}{B}\right)\right).
\end{split}
\end{equation}
}
We shall indicate the argument in the case $p_3=p$ and $c_3=c$. ($p_3=1$ or $c_3=1$ works the same way and we get a better bound). In this case (see \eqref{defintion of X and Y})
$$
X= pqc^2T \text{ and } Y= \frac{pqc^2T}{r}.
$$ We shall apply Voronoi summation to the sum over $n,m$. For this purpose we have to rewrite the congruence condition in terms of additive characters. Without getting into complete details (the steps are identical to the ones carried out in this section till \eqref{Voronoi transformed S1 r,c}), we get the following ``inequality":
\begin{multline}\label{voronoi transformed S_2(A,B)}
\left(\sum_{\substack{n,m}}{\overline{\lambda_f}(n)\overline{\lambda_g}(m) \delta[rm \equiv n (pc) ]} W_2\left(\frac{n}{A},\frac{m}{B}\right) \right)\ll \cr \sum_{\pm,\pm}\left(\sum_{\substack{n,m}}{\overline{\lambda_f}(n)\overline{\lambda_g}(m) \delta[m \equiv qrn (pc) ]} \right)W_2^{\pm,\pm}\left(\frac{n}{A'},\frac{m}{B'}\right),
\end{multline}
where 
\begin{equation}
W_2^{\pm,\pm}\left(\xi,\eta\right)= \frac{AB}{\sqrt{q}(pc)^2}\int_0^\infty\int_0^\infty{W_2(x,y)J^{\pm}_f\left(4\pi \sqrt{\xi x}\right) J^{\pm}_g\left(4 \pi \sqrt{\eta y}\right)dxdy},
\end{equation}
\begin{equation*}
A' = \frac{(pc)^2}{A}, \text{ and } B'= \frac{q(pc)^2}{B}.
\end{equation*} Using the bounds for derivatives of $W_2$ \eqref{bound on derivatives of W2} and Lemma \ref{Bounds for integral transforms}, we conclude that for all $j \geq 0$,
\begin{equation}\label{Bound on W2}
W_2^{\pm,\pm}\left(\frac{n}{A'},\frac{m}{B'}\right) \ll_j \frac{AB}{\sqrt{q}(pc)^2}\min\left\{1, \left(\frac{A'}{n}\right)^{-j},\left(\frac{B'}{m}\right)^{-j},\left( \frac{A'B'}{mn}\right)^{-j} \right\}.
\end{equation}
Thus
\begin{multline}
\left(\sum_{\substack{n,m}}{\overline{\lambda_f}(n)\overline{\lambda_g}(m) \delta[rm \equiv n (pc) ]} W_2\left(\frac{n}{A},\frac{m}{B}\right) \right) \ll\\
 \frac{AB}{\sqrt{q}(pc)^2}\sum_{\substack{n \leq (pq)^\epsilon A'\\ m \leq (pq)^\epsilon B'} }{|\lambda_f(n)||\lambda_g(m)| \delta[m \equiv qrn (pc) ]}.
\end{multline}
We use Lemma \ref{Bounds for convolution of fourier coefficients on AP} to bound the right hand side. Putting this back into \eqref{S_3 r,c q_2=q}, we get
\begin{equation}
S_3^{r,c,q_2=q}(A,B) \ll (pq)^\epsilon\frac{rp}{cT^2} \sqrt{q}pc\left(1+ \sqrt{\frac{A}{pc}} + \sqrt{\frac{B}{qpc}} + \frac{\sqrt{AB}}{pc\sqrt{q}}\right).
\end{equation}
 Using this bound in \eqref{S2 in terms of S3} and the upper bounds $A \leq (pq)^\epsilon X$ and $B \leq (pq)^\epsilon Y$, we get
\begin{equation}
S_2^{r,c,q_2=q}(N) \ll (pq)^\epsilon \frac{p^2q c \sqrt{r}}{T}.
\end{equation}
Putting this back into \eqref{S_1^q_2=q}, we get
\begin{align}
S_1^{dual,q_2=q}(N) &\ll  \frac{1}{pq|\mathcal{L}| |\mathcal{C}|}\sum_{r}\sum_{c \in \mathcal{C}}{\frac{|\alpha_r|}{c}}{|S_2^{r,c,q_2=q}(N)|}\\
&\ll (pq)^\epsilon \frac{NL}{q}.
\end{align}
The bound for $S_1^{dual,p_2=p}$ and $S_2^{dual,c_2=c}$ can be shown along the same lines.

\end{proof}

Since $C = 10L^2$ and $L \leq q^{1/2}$, the bounds in Lemma \ref{Bound for q_1=q} and Lemma \ref{Bound for q_2=q} imply that
\begin{align}\label{bound for boundary terms}
S_1^{dual,q_1=q}(N), S_1^{dual,q_2=q}, S_1^{dual,c_1 =c}, S_1^{dual, c_2=c} \ll (pq)^\epsilon\frac{\sqrt{Npq}}{\sqrt{L}}.
\end{align}
This allows us to get rid of these boundary terms. If the size of $p$ was large, say $p \approx q$, then we could have gotten rid of $p_1 =p$ and $p_2=p$ also. But since $p$ could be very small, this is not possible. Restricting to $q_3=q$ and $c_3=c$, we rewrite \eqref{S2 in terms of S3} as
{\small
\begin{equation}\label{S2 without boundary terms}
\begin{split}
S_2^{r,c}(N) &= \sum_{\substack{p_1p_2p_3 =p} } \sum_{\substack{(\rho_1, A \leq (pq)^\epsilon X) \\ (\rho_2, B \leq (pq)^\epsilon Y) } }{ {\mu(p_2) \frac{rN^2}{(p_2)^2p_3qc\sqrt{p_1}}} }\times\\
&\left(\sum_{\substack{n,m}}{\overline{\lambda_f}(n)\overline{\lambda_g}(m) \delta[p_1rm \equiv qn (p_3qc) ]} W_2\left(\frac{n}{A},\frac{m}{B}\right)\right).
\end{split}
\end{equation}
}
At this point we rewrite the congruence $p_1rm \equiv n (p_3qc)$ in \eqref{S2 without boundary terms} as an equality, 
$$n=p_1rm+ sp_3qc.$$
The value of $s$ can be negative. Let $\vec{p} = (p_1,p_2,p_3)$ be any factorization of $p$. Since $n \leq A$ and $m \leq B$, we get
\begin{equation}\label{definition of S}
|s| \leq S_{\vec{p}} := \max\left\{\frac{A}{p_3qc}, \frac{p_1r B}{p_3qc} \right\}.
\end{equation}
Thus
{\small
\begin{equation}\label{S2 with congruence replaced by equality}
\begin{split}
S_2^{r,c}(N) &=  \sum_{\substack{p_1p_2p_3 =p} } \sum_{\substack{(\rho_1, A \leq (pq)^\epsilon X) \\ (\rho_2, B \leq (pq)^\epsilon Y) } }{ {\mu(p_2) \frac{rN^2}{\sqrt{p_1}p_2^2p_3qc}} }\times\\
&\left(\sum_{|s| \leq S_{\vec{p}}}\sum_{\substack{m}}{\overline{\lambda_f}(p_1rm+ sp_3qc)\overline{\lambda_g}(m)W_2\left(\frac{p_1rm+ sp_3qc}{A},\frac{m}{B}\right) } \right).
\end{split}
\end{equation}
}

When $s=0$ in the equation above the sum over $r,c$ collapses. Hence, we need a separate treatment of the terms with $s=0$. When $s=0$, the inner sum is of the form
$$
S(B) = \sum_{m \leq B}{\overline{\lambda_f}(m)\overline{\lambda_f}(m)},
$$
which appears to be the conjugate of the sum we that we started with in \eqref{Definitio of S(N)}. But, since the length of the $m$-sum, $B = Y = \frac{(pqc)^2}{N} \gg pqC^2$, is greater than the square root of the arithmetic conductor of $L(f \times g,s)$, we can get some cancellation in this sum. We exhibit a satisfactory bound for the contribution of the diagonal terms i.e $s=0$, in the next Lemma. 
\begin{lemma}\label{Bounds on S2 s=0}
Let the contribution of the terms with $s=0$ in \eqref{S2 with congruence replaced by equality} to \eqref{S1 dual in terms of S2r,c} be $S_1^{dual,s=0}(N)$ i.e
{\small
\begin{align}
S_1^{dual,s=0}(N) =& \frac{1}{pq|\mathcal{L}| |\mathcal{C}|}\sum_{r}\sum_{c \in \mathcal{C}}{\frac{\alpha_r}{c}}{S_2^{r,c,s=0}(N)}\\
=& \frac{1}{pq|\mathcal{L}| |\mathcal{C}|}\sum_{r}\sum_{c \in \mathcal{C}}{\frac{\alpha_r}{c}}\sum_{\substack{p_1p_2p_3 =p} } \sum_{\substack{(\rho_1, A \leq (pq)^\epsilon X) \\ (\rho_2, B \leq (pq)^\epsilon Y) } } \label{definition of dual S1 s=0}{\mu(p_2) \frac{rN^2}{\sqrt{p_1}p_2^2p_3qc}}\\
&  \times \left(\sum_{\substack{m}}{\overline{\lambda_f}(p_1rm)\overline{\lambda_g}(m)W_2\left(\frac{p_1rm}{A},\frac{m}{B}\right) } \right).
\end{align}
}
Then
\begin{equation}\label{Bound for S1 s=0}
S_1^{dual,s=0}(N) \ll (pq)^\epsilon \frac{\sqrt{Npq}}{\sqrt{L}}.
\end{equation}
\end{lemma}
\begin{proof}
Among the three choices i.e $p_1=p$, $p_2=p$, or $p_3=p$ in \eqref{definition of dual S1 s=0}, $p_3=p$ makes the largest contribution. We shall exhibit the bound \eqref{Bound for S1 s=0} in this case. It is straightforward to handle the other two cases using the same method.
{\small
\begin{equation}\label{S1 s=0, p_3=p}
\begin{split}
S_1^{dual,p_3=p,s=0}(N) \ll \frac{N^2}{(pqC)^2|\mathcal{L}| }\sum_{r}{r| \alpha_r| }\sum_{\substack{(\rho_1, A \leq (pq)^\epsilon X) \\ (\rho_2, B \leq (pq)^\epsilon Y) } }\left(\sum_{\substack{m}}{\overline{\lambda_f}(rm)\overline{\lambda_g}(m) } W_2\left(\frac{rm}{A},\frac{m}{B}\right) \right) .
\end{split}
\end{equation}
}
Using the Rankin-Selberg bound \eqref{rankin selberg bound}, the sum over $m$ can be rewritten as follows:
\begin{align}
\sum_{\substack{m}}{\overline{\lambda_f}(rm)\overline{\lambda_g}(m) } &= \overline{\lambda_f}(r)\sum_{\substack{(m,l)=1}}{\overline{\lambda_f}(m)\overline{\lambda_g}(m) }  + \sum_{\substack{l|m}}{\overline{\lambda_f}(rm)\overline{\lambda_g}(m) } \\
&= \overline{\lambda_f}(r)\sum_{\substack{m}}{\overline{\lambda_f}(m)\overline{\lambda_g}(m) } +O\left((pq)^\epsilon|\sigma_f(rl)||\lambda_g(l)|)\frac{B}{l} \right),
\end{align}
where 
\begin{equation}\label{definition of sigma f}
\sigma_f(n) = \sum_{d|n}{|\lambda_f(d)|},
\end{equation} and $l$ has been defined in $\eqref{definition of amplifier}$. Noting that $B \leq (pq)^\epsilon Y = (pq)^\epsilon\frac{(pqc)^2}{rN} $, we use Lemma \ref{Polya-Vinogradov type inequality} to bound the $m$-sum. This gives
$$
\sum_{\substack{m}}{\overline{\lambda_f}(rm)\overline{\lambda_g}(m) }W(m/Y) \ll (pq)^\epsilon\left(|\lambda_f(r)|pqC\sqrt{\frac{pq}{rN} } + |\sigma_f(rl)||\lambda_g(l)|\frac{(pqc)^2}{rNL}\right).
$$
Substituting this bound for the $m$-sum in \eqref{S1 s=0, p_3=p} and using $H_{1/4}$  (\ref{H theta}) to bound the second term, we get
\begin{equation}
S_1^{dual,p_3=p,s=0}(N) \ll (pq)^\epsilon \frac{N}{\sqrt{L}} \ll (pq)^\epsilon \frac{\sqrt{Npq}}{\sqrt{L}}.
\end{equation}

\end{proof}

As the structure of the sum \eqref{S2 with congruence replaced by equality} is different depending on whether $p_1=p$, $p_2=p$, and $p_3=p$, we have to treat then separately. Since the size of $Y$ (see \eqref{defintion of X and Y}) depends on $r$, it is convenient to dyadically divide the sum over $r$ in \eqref{S1 dual in terms of S2r,c}. Exchanging the sum over $\{r,c\}$ and $m$ in \eqref{S1 dual in terms of S2r,c} , \eqref{S2 with congruence replaced by equality}, we have
\begin{align}\label{S1 dual after eliminating boundary terms}
S_1^{dual}(N) &=\frac{1}{pq|\mathcal{L}| |\mathcal{C}|}\sum_{R =2^\nu \leq L^2}\sum_{r \in [R,2R]}\sum_{c \in \mathcal{C}}{\frac{\alpha_r}{c}}{S_2^{r,c}(N)}\\
&= \sum_{p_1p_2p_3=p}{\mathcal{S}_{(p_1,p_2,p_3)}(N)} + O\left((pq)^\epsilon\frac{\sqrt{Npq}}{\sqrt{L}}\right),
\end{align}
where
{\small
\begin{align}\label{S p_1,p_2,p_3}
\mathcal{S}_{\vec{p}} = \frac{1}{pq|\mathcal{L}| |\mathcal{C}|} \frac{N^2}{\sqrt{p_1}p_2^2p_3q}&\sum_{R =2^\nu \leq L^2}\sum_{\substack{(\rho_1, A \leq (pq)^\epsilon X_{\vec{p}}) \\ (\rho_2, B \leq (pq)^\epsilon Y_{\vec{p}}) } }\sum_{m \leq B}{\overline{\lambda_g}(m) } \times\\
&{\left(\sum_{r \in [R,2R]}\sum_{c \in \mathcal{C}}\sum_{|s| \neq 0 ,\leq S_{\vec{p}}}{\frac{r\alpha_r}{c^2} }{\overline{\lambda_f}(p_1rm +p_3qcs)}W_2\left(\frac{p_1rm +p_3qcs}{A},\frac{m}{B}\right) \right)},\\
\end{align}
}
\begin{align}
X_{\vec{p}} &= \frac{p_1(p_2p_3qc)^2}{N}, Y_{\vec{p}} =\frac{(p_2p_3qc)^2}{RN}, \text{ and } S_{\vec{p}}= \max\left\{\frac{A}{p_3qc}, \frac{p_1r B}{p_3qc} \right\}.\label{definition of Xp,Yp,Sp} 
\end{align}
For example: if $\vec{p}=(1,1,p)$, then
{\small
\begin{align}\label{S p_3=p}
\mathcal{S}_{(1,1,p)} =\frac{1}{pq|\mathcal{L}| |\mathcal{C}|}\frac{N^2}{pq}&\sum_{R =2^\nu \leq L^2}\sum_{\substack{(\rho_1, A \leq (pq)^\epsilon X_3) \\ (\rho_2, B \leq (pq)^\epsilon Y_3) } }\sum_{m \leq B}{\overline{\lambda_g}(m) }\times\\
&{\left(\sum_{r \in [R,2R]}\sum_{c \in \mathcal{C}}\sum_{|s| \neq 0, \leq S_3}{\frac{r\alpha_r}{c^2}}{\overline{\lambda_f}(rm +pqcs)}W_2\left(\frac{rm +pqcs}{A},\frac{m}{B}\right) \right)},
\end{align}
}
\begin{align}
X_3 &=\frac{(pqc)^2}{N},Y_3 = \frac{(pqc)^2}{RN}, \text{ and } S_3 = \max\left\{\frac{A}{pqc}, \frac{RB}{pqc} \right\}\label{definition of X3,Y3,S3}.
\end{align}

We plan to eliminate the oscillation due to $\lambda_g(m)$ in \eqref{S p_1,p_2,p_3} using Cauchy-Schwarz inequality, leading us to a shifted convolution problem. Our bounds for the shifted convolution problem are not optimal if the weight functions don't have small Sobolev norms (see the dependence on $K_1$ and $K_2$ in \eqref{bound for SCP non-zero shift}). If $A$ and $p_1rB$ are not of the same size, then the weight functions appearing in the shifted convolution problem have large derivatives. In order to circumvent this issue, we separate the weight function $W_2(x,y)$ as a product of a function in $x$ and $y$. There are many standard ways of doing this. Since $W_2$ is compactly supported smooth function with support contained in $[1/2,1] \times [1/2,1]$, using Fourier inversion formula, we can see that
$$
W_2(x,y) = \int_{z \in \mathbb{R}}{\widehat{W_2}(x,z)e(-zy)dz },
$$
where $\widehat{W_2}(x,z)$ is the Fourier transform of $W_2$ with respect to the second variable i.e
$$
\widehat{W_2}(x,z) = \frac{1}{2\pi}\int_{\eta \in \mathbb{R}}{W_2(x,\eta)e(-\eta z)d\eta }.
$$
Using the estimate \eqref{bound on derivatives of W2} for the derivatives of $W_2$ and integration by parts, we see that for all $N,j \geq 0$
\begin{equation}\label{bound for derivatives of Fourier transform of W2}
\left(\frac{d}{dx}\right)^j\widehat{W_2}(x,z) \ll_{N,j} \min\left\{1, \frac{1}{|z|^N} \right\}.
\end{equation}
Choosing $N = 2019 \epsilon^{-1}$ in \eqref{bound for derivatives of Fourier transform of W2}, we have
$$
W_2(x,y) = \int_{|z| \leq (pq)^\epsilon}{\widehat{W_2}(x,z)e(-zy)dz } + O\left(\frac{1}{(pq)^{2018}}\right).
$$
Substituting this expression in \eqref{S p_1,p_2,p_3}, we get
{\small
\begin{align}\label{S p_1,p_2,p_3 weight function modified}
\mathcal{S}_{\vec{p}} &= \frac{1}{pq|\mathcal{L}| |\mathcal{C}|} \frac{N^2}{\sqrt{p_1}p_2^2p_3q}\int_{|z| \leq (pq)^\epsilon }  \sum_{R =2^\nu \leq L^2}\sum_{\substack{(\rho_1, A \leq (pq)^\epsilon X_{\vec{p}}) \\ (\rho_2, B \leq (pq)^\epsilon Y_{\vec{p}}) } }\sum_{m \leq B}{\overline{\lambda_g}(m) e\left(\frac{-zm}{B} \right)} \times\\
& {\left(\sum_{r \in [R,2R]}\sum_{c \in \mathcal{C}}\sum_{|s| \neq 0 ,\leq S_{\vec{p}}}{\frac{r\alpha_r}{c^2} }{\overline{\lambda_f}(p_1rm +p_3qcs)}W_{2,z}\left(\frac{p_1rm +p_3qcs}{A}\right) \right)} + O((pq)^{-1000}) ,
\end{align}
}
where
\begin{equation}
W_{2,z}(x) := \widehat{W_2}(x,z).
\end{equation}
$W_{2,z}$ is a smooth function with compact support contained in $[1/2,1]$ with derivatives satisfying \eqref{bound for derivatives of Fourier transform of W2}. Applying Cauchy-Schwarz inequality on the sum over $m$ and taking the supremum over $z$ (using the Rankin-Selberg bound \eqref{rankin selberg bound}), we have
{\small
\begin{equation}\label{S_p after cauchy}
\begin{split}
(\mathcal{S}_{\vec{p}})^2 &\ll (pq)^\epsilon \left(\frac{N}{pq}\right)^4 \frac{p_1}{(p_2|\mathcal{L}| |\mathcal{C}| )^2}\sup_{z}\sum_{R =2^\nu \leq L^2}\sum_{\substack{(\rho_1, A \leq (pq)^\epsilon X_{\vec{p}}) \\ (\rho_2, B \leq (pq)^\epsilon Y_{\vec{p}}) } }{\frac{R^2B}{C^4} } \sum_{r,r' \in [R,2R]}{ \gamma_r \overline{\gamma_{r'}} } \sum_{|t|, |t'| \neq 0 \leq T_{\vec{p}} }{\beta_t \beta_{t'}}\\
&\left(\sum_{\substack{m}}{\overline{\lambda_f}(p_1rm+ p_3qt)\lambda_f(p_1r'm+ p_3qt') } W_{2,z}\left(\frac{p_1rm+ p_3qt}{A}\right)\overline{W_{2,z}\left(\frac{p_1r'm+ p_3qt'}{A}\right) } \right),
\end{split}
\end{equation}
}
where
\begin{align}\label{definition of gamma and beta}
\gamma_r := \frac{r}{R} \alpha_r \leq 2| \alpha_r|,\\
\beta_t  := \sum_{c \in \mathcal{C} , c|t} {\frac{C^2}{c^2} } \leq 4 d(t),
\end{align}
and 
\begin{equation} \label{definition of Tp}
T_{\vec{p}} = 2CS_{\vec{p}} = \max\left\{\frac{A}{p_3q}, \frac{p_1r B}{p_3q} \right\}.
\end{equation}
The sum over $m$ leads us to our shifted convolution problem. The important feature to observe above is that when $p_3=p$ (this is the important case for large $p$), the shifts (defined in \ref{definition of SCP}) are multiples of $p$, which is the level of the automorphic form $f$. On a first reading, the reader can skip the next section assuming the contents of Theorem \ref{SCP in ab coprime to p}, without losing continuity.

\section{Shifted convolution problem \label{shifted convolution problem} }

We recall the $\delta$ method of Duke, Friedlander, and Iwaniec \cite{Duke1994}. We use the version due to Heath-Brown in \cite{heath1996new}. The following Lemma is from Munshi's paper \cite[Lemma 23]{munshi2017subconvexity}. In this section we shall use the letter $q$ to denote the moduli in the circle method. This should not be confused with the level of the modular form $g$ in the previous sections.

\begin{lemma}
For any $Q \geq 1$, there exists a positive constant $c_0$ and a smooth function $h(x,y)$ defined on $(0, \infty) \times \mathbb{R}$, such that
\begin{equation}\label{delta method}
\delta(n,0) = \frac{c_0}{Q}\sum_{q=1}^{\infty}{\frac{1}{q} }\sum_{\gamma (q)^*}{e\left(\frac{\gamma n}{q} \right)h\left(\frac{q}{Q},\frac{n}{Q^2} \right)}.
\end{equation}
The constant $c_0$ satisfies $c_0 = 1+ O_A(Q^{-A})$ for any $A >0$. Moreover $h(x,y)\ll x^{-1}$ for all $y$, and $h(x,y)$ is non-zero only for $x \leq \max\{1,2y\}.$ If $|y| \leq \frac{x}{2}$ and $x \leq 1$, then 
$$
\frac{\partial}{\partial y}h(x,y) =0.
$$
Furthermore for all $N ,j\geq 0$ and $x \leq 1$, h satisfies
\begin{equation}\label{bound for h}
y^j\frac{\partial^{j}}{\partial y^j}h(x,y) \ll_{N} x^N + \min \{1, (x/|y|)^N \} \ll 1. 
\end{equation}
\end{lemma}
\begin{remark}
We have normalized $h(x,y)$ differently from \cite{munshi2017subconvexity}. We have multiplied the $h$ in Munshi's paper by $x$ (see Lemma 23, \cite{munshi2017subconvexity}).
\end{remark}

\begin{lemma}
For any $0 \leq x \leq 1$ and $h$ as in \eqref{delta method},
\begin{equation}\label{bound for integral of h}
\int_{\mathbb{R}}|h(x,y)|dy \ll x.
\end{equation}
\end{lemma}

\begin{proof}
Using \eqref{bound for h}
\begin{align}
\int_{\mathbb{R}}{|h(x,y)|dy} &= \int_{|y| \leq x}{|h(x,y)|dy} + \int_{|y| \geq x}{|h(x,y)|dy} \\
& \ll x + \int_{|y| \geq x}{\frac{x^2}{|y|^2}dy}\\
& \ll x.
\end{align}
\end{proof}

In practice to detect to $n =0$ for a sequence of integers in the range $[-X,X]$, it is logical to choose $Q = 2\sqrt{X}$, so that in the generic range for $q$ there is no oscillation of the weight function $h(x,y)$.

\begin{definition}\label{definition of SCP}
Let $K_1,K_2,M_1,M_2 \geq 1$, $W: [1/2,3] \times [1/2,3] \to \mathbb{C}$ be a compactly supported smooth function satisfying 
\begin{equation}\label{definition of K1 and K2}
W^{(i,j)}(x,y) \ll K_1^iK_2^j,
\end{equation}
for all $i,j \geq 0$.
Let  $a,b,c,d$ be integers and $a,b \neq 0$. If $f$ and $g$ are cuspidal automorphic forms (modular or Maass) of arbitrary level and nebentypus we define the shifted convolution sum $S$ as,
\begin{equation}\label{Equation defining SCP}
S_{f,g}(a,b,c,d,M_1,M_2) = \sum_{m=1}^\infty{\lambda_f(am+c)\lambda_g(bm+d)W_1\left(\frac{am+c}{M_1},\frac{bm+d}{M_2}\right) }.
\end{equation}
We define $(ad-bc)$ to be the ``shift'' of the sum $S_{f,g}$.
\end{definition}

We first establish a bound for $S_{f,g}(a,b,c,d,M_1,M_2)$ in the case $(ab,p)=1$.

\begin{remark}
We have assumed the primality of $p$, as this is the only case we shall need in this paper. The proof can be easily extended to include any $p \geq 1$. We can improve the bounds in this Lemma by using Spectral theory for $GL(2)$ to bound a sum of Kloosterman sums (see equation \ref{Voronoi Transformed SCP}). 
\end{remark}

\begin{proof}[Proof of Theorem \ref{SCP in ab coprime to p}]
The first bound in \eqref{bound for SCP rankin-selberg} follows from the Rankin-Selberg bound \eqref{rankin selberg bound}. The second bound follows from hypothesis $H_\theta$ \eqref{H theta}.

 We separate the oscillation of $f$ and $g$ in \eqref{Equation defining SCP} by introducing a $\delta$ symbol i.e
\begin{equation} \label{Expression for SCP}
S = \sum_{\substack{n_1 \equiv c(a)\\n_2 \equiv d(b)}}{\lambda_f(n_1)\lambda_g(n_2) \delta\left(\frac{n_1-c}{a}=\frac{n_2-d}{b} \right)} W_1\left(\frac{n_1}{M_1},\frac{n_2}{M_2}\right).
\end{equation}
We imitate Munshi's ideas (see \cite{munshi2017subconvexity}) to factor the $\delta$ symbol as a congruence mod $p$,
 \begin{equation}\label{congruence mod p}
\frac{n_1-c}{a} \equiv \frac{n_2-d}{b} (p),
 \end{equation}
 followed by the equality 
 \begin{equation}\label{factored equality}
\frac{n_1-c}{ap} - \frac{n_2-d}{bp} = 0.
 \end{equation}

 If $|ad-bc| > 3(|b|M_1+|a|M_2) $ then the shifted convolution sum $S_{f,g}$ is $0$, as the summation is empty. Thus the difference  $\left|\frac{n_1-c}{ap} - \frac{n_2-d}{bp} \right|$  is bounded by $\frac{X}{2}$, where
 \begin{align}\label{defintion of X}
 X= 12\left(\frac{M_1}{|a|} + \frac{M_2}{|b|}\right).
\end{align}
Let us define
\begin{equation}
Q= \sqrt{X/p}.
\end{equation}
  We assume $Q \geq 1$, otherwise \eqref{bound for SCP rankin-selberg} implies \eqref{bound for SCP non-zero shift} and $\eqref{bound for SCP shift multiple of level}$. We pick up the congruence \eqref{congruence mod p} using additive characters modulo $p$ and the equality \eqref{factored equality} using the Duke, Friedlander, and Iwaniec circle method \eqref{delta method}, with $Q$ as above. Thus
{\small
 \begin{equation}\label{factored delta symbol}
 \begin{split}
\delta\left(\frac{n_1-c}{a}=\frac{n_2-d}{b} \right) = \frac{c_0}{pQ}\sum_{q=1}^{\infty}{\frac{1}{q}}\sum_{\alpha(p)}&\sum_{\gamma (q)^*}{e\left(\frac{\alpha (n_1 -c)}{ap} -\frac{\alpha (n_2 -d)}{bp}   \right)}\\
&e\left(\frac{\gamma (n_1 -c)}{apq} -\frac{\gamma (n_2 -d)}{bpq}   \right)h\left(\frac{q}{Q},\frac{\frac{n_1-c}{ap}-\frac{n_2-d}{bp}}{Q^2} \right).
 \end{split}
 \end{equation}
}
Substituting this into \eqref{Expression for SCP}, we get
{\small
\begin{equation}
\begin{split}
S= \frac{c_0}{abpQ}\sum_{q=1}^{\infty}{\frac{1}{q}}\sum_{ \substack{\alpha(ap) \\ \beta(bp) \\ \alpha \equiv \beta (p)} }&\sum_{\gamma(q)^*}\sum_{\substack{n_1 \\n_2 }}{\lambda_f(n_1)\lambda_g(n_2)}{e\left(\frac{\alpha (n_1 -c)}{ap} -\frac{\beta (n_2 -d)}{bp}   \right)}\\
&e\left(\frac{\gamma (n_1 -c)}{apq} -\frac{\gamma(n_2 -d)}{bpq}   \right)h\left(\frac{q}{Q},\frac{\frac{n_1-c}{ap}-\frac{n_2-d}{bp}}{Q^2} \right) W\left(\frac{n_1}{M_1},\frac{n_2}{M_2}\right).
\end{split}
\end{equation}
}
We have expanded the congruence $n_1 \equiv c(a)$ and $n_2 \equiv d(b)$ using additive characters and combined the frequency mod $p$ using the  Chinese remainder theorem. Combining the frequency $\gamma$ modulo $q$ with $\alpha(ap)$ and $\beta(bp)$, we get
{\small
\begin{equation} \label{Transformed SCP}
\begin{split}
S= \frac{c_0}{abpQ}\sum_{q=1}^{\infty}{\frac{1}{q}}\sum_{ \substack{\alpha(apq) \\ \beta(bpq) \\ \alpha \equiv \beta (pq) \\ (\alpha,q)=1} }\sum_{\substack{n_1 \\n_2 }}{\lambda_f(n_1)\lambda_g(n_2)}{e\left(\frac{\alpha (n_1 -c)}{apq} -\frac{\beta (n_2 -d)}{bpq}   \right)}W_1\left(\frac{n_1}{M_1},\frac{n_2}{M_2} \right),
\end{split}
\end{equation}
}
where
\begin{equation}
W_1\left(\frac{n_1}{M_1},\frac{n_2}{M_2}\right) = h\left(\frac{q}{Q},\frac{\frac{n_1-c}{ap}-\frac{n_2-d}{bp}}{Q^2} \right) W\left(\frac{n_1}{M_1},\frac{n_2}{M_2}\right).
\end{equation}
Using \eqref{bound for h} and \eqref{definition of K1 and K2}, we see that for all $i,j \geq 0$
\begin{equation}\label{bound for derivatives of W1}
W_1^{(i,j)}(x,y) \ll K_1^iK_2^j\left(\frac{Q}{q}\right)^{i+j}.
\end{equation}

We plan to apply the Voronoi summation formula \eqref{voronoi summation} to the $n_1$ and $n_2$ sum in \eqref{Transformed SCP}. The modulus in Voronoi summation depends on gcd$(\alpha,apq)$ and gcd$(\beta, bpq)$. Since $\alpha$ and $\beta$ are coprime to $q$, the gcd$(\alpha,apq)$ and $(\beta,bpq)$ are $(\alpha,ap)$ and $(\beta,bp)$ respectively. We will assume that $(\alpha,ap) = (\beta,bp) =1$ and $(q,\text{Rad}(ab) )=1$, to simplify notation. ($\text{Rad}(a)$ is the radical of $a$, defined to be $\text{Rad}(a) = \prod_{p|a}{p^\infty}$.) We shall indicate the changes required to handle the general case towards the end of the proof. The same bound established below holds below for the other cases too. Let us define
{\small
\begin{equation}\label{S1,1,1}
\begin{split}
S^{1,1,1}= \frac{c_0}{abpQ}\sum_{(q,ab)=1}{\frac{1}{q}}\sum_{ \substack{\alpha(apq)* \\ \beta(bpq)* \\ \alpha \equiv \beta (pq) } }\sum_{\substack{n_1 \\n_2 }}{\lambda_f(n_1)\lambda_g(n_2)}{e\left(\frac{\alpha (n_1 -c)}{apq} -\frac{\beta (n_2 -d)}{bpq}   \right)}W_1\left(\frac{n_1}{M_1},\frac{n_2}{M_2}\right).
\end{split}
\end{equation}
}
The ${(1,1,1)}$, in the superscript of $S$, refers to the greatest common divisors $(\alpha,ap)$, $(\beta,bp)$, and $(q,\text{Rad}(ab) )$ respectively. Applying the Voronoi summation formula to $n_1$ and $n_2$ in \eqref{S1,1,1}, we get
{\small
\begin{equation}\label{Voronoi transformed SCP before simplification}
\begin{split}
S^{1,1,1}= \frac{c_0}{abpQ}\sum_{(q,ab)=1}{\frac{1}{q}}\sum_{ \substack{\alpha(apq)* \\ \beta(bpq)* \\ \alpha \equiv \beta (pq) } }\sum_{\substack{n_1 \\n_2 }}{\overline{\lambda_{f}} (n_1)\overline{\lambda_{g}}(n_2)}{e\left(\frac{\overline{\alpha} n_1 }{apq} -\frac{\overline{\beta} n_2}{bpq}   \right)e\left(-\frac{\alpha c }{apq} +\frac{\beta d}{bpq}   \right) }\\
W_2\left(\frac{M_1n_1}{(apq)^2},\frac{M_2n_2}{(bpq)^2}\right) + \sum_{\substack{ (-,+) \\ (+,-) \\ (-,-)}}{\dots},
\end{split}
\end{equation}
}
where
\begin{equation}\label{definition of W2 in SCP}
W_2\left(\xi,\eta\right)= \frac{M_1M_2}{ab(pq)^2}\int_0^\infty\int_0^\infty{W_1(x,y)J^+_f\left(4\pi \sqrt{\xi x}\right) J^+_g\left(4 \pi \sqrt{\eta y}\right)dxdy}.
\end{equation}
The summation $\sum_{\substack{ (-,+) \\ (+,-) \\ (-,-)}}{\dots}$, refers to the three other terms coming from the $\pm$ terms on the right hand side of the Voronoi formula \eqref{voronoi summation}. They have the same structure and can be handled similarly. 
Lemma \ref{Bounds for integral transforms} and the bound \eqref{bound for derivatives of W1} implies that for all $N,M \geq 0$
$$
W_2\left(\xi,\eta\right) \ll \left(\frac{(K_1Q)^2}{q^2 \xi} \right)^{-N} \left(\frac{(K_2Q)^2}{q^2 \eta}\right)^{-M}.
$$
Thus, choosing $N$ and $M$ sufficiently large, we can restrict the summation over $n_1$ and $n_2$ in \eqref{Voronoi transformed SCP before simplification} to 
\begin{equation}\label{definition of N1 and N2}
n_1 \ll (pX)^\epsilon\frac{(K_1apQ)^2}{M_1} = N_1 \text{ and } n_2 \ll (pX)^\epsilon\frac{(K_2bpQ)^2}{M_2} =N_2.
\end{equation}
In the complementary range for $n_1$ and $n_2$ we shall use standard estimates for Bessel functions \cite[Lemma C.2]{kowalski2002rankin} to bound $W_2$ as follows.
We can assume $bM_1 \leq aM_2$, without loss of generality. Making the substitution $x \rightsquigarrow x/M_1$ and $y \rightsquigarrow y/M_2$ in the integral \eqref{definition of W2 in SCP}, we have
{\small
\begin{equation}
\begin{split}
W_2\left(\xi,\eta\right) = \frac{1}{ab(pq)^2}\int_{M_1/2}^{3M_1}\int_{M_2/2}^{3M_2}{W\left(\frac{x}{M_1},\frac{y}{M_2}\right)h\left(\frac{q}{Q},\frac{\frac{x-c}{a}-\frac{y-d}{b}}{X} \right)J^+_f\left(4\pi \sqrt{\frac{\xi x}{M_1} }\right)}\\
 J^+_g\left(4 \pi \sqrt{\frac{\eta y}{M_2} }\right)dxdy.
\end{split}
\end{equation}
}
{\small
\begin{align}
W_2(\xi,\eta)&\ll \frac{1}{(\xi \eta)^{1/4}(abpq)^2}\int_{bM_1/2}^{3bM_1}\int_{aM_2/2}^{3aM_2}{\left|W\left(\frac{x}{bM_1},\frac{y}{aM_2}\right)\right| \left|h\left(\frac{q}{Q},\frac{x-y + ad-bc}{abX} \right|\right)dxdy}\\
&\ll \frac{1}{(\xi \eta)^{1/4}(abpq)^2}\int_{bM_1/2}^{3bM_1}{\left|W\left(\frac{x}{bM_1},\frac{x-u+h}{aM_2}\right)\right|}\int{\left|h\left(\frac{q}{Q},\frac{u}{abX} \right)\right|dudx},
\end{align}
}
where $h= ad-bc$ is the shift. Using estimate \eqref{bound for integral of h} to bound the integral over $u$, we get the bound
{\small
\begin{align}\label{bound on W2 xi eta}
W_2\left(\xi,\eta\right) &\ll \frac{M_1M_2q}{(\xi \eta)^{1/4}ab(pq)^2Q}.
\end{align}	
}
Using the Chinese remainder theorem to split the exponential sum modulo $pq$ and $ab$, equation \eqref{Voronoi transformed SCP before simplification} can be rewritten as
{\small
\begin{equation} \label{Voronoi Transformed SCP}
\begin{split}
S^{1,1,1}= \frac{c_0}{abpQ}\sum_{\substack{(q,ab)=1 \\q \leq Q }}{\frac{1}{q}}&\sum_{\substack{n_1 \leq N_1 \\n_2 \leq N_2 }}{\overline{\lambda_f}(n_1)\overline{\lambda_g}(n_2)}S(\overline{ab}(bn_1-an_2);\overline{ab}(ad-bc);pq)\\S(\overline{pq}n_1;-\overline{pq}c;a)
&S(-\overline{pq}n_2;\overline{pq}d;b)W_2\left(\frac{M_1n_1}{(apq)^2},\frac{M_2n_2}{(bpq)^2}\right),
\end{split}
\end{equation}
}
where $S(m,n,c) = \sum_{\alpha(c)*}{e\left(\frac{\alpha m + \overline{\alpha}n}{c} \right)},$ is the Kloosterman sum. We shall use the Weil bound
\begin{equation}\label{Weil bound}
S(m,n,c) \ll (m,n,c)^{1/2}c^{1/2+\epsilon}.
\end{equation}
The crucial point here is that the shift $(ad-bc)$ will be a multiple of $p$ in our application. In this case, the Kloosterman sum modulo $p$ becomes a Ramanujan sum \eqref{Ramanujan sum}. If $n \equiv 0 (c)$, then
$$
S(m,n,c) =S(m,0,c) = \mathfrak{r}_c(m) \ll (m,c)c^{\epsilon}.
$$
This allows us to save an additional $\sqrt{p}$ in the Kloosterman sum modulo $p$. We first consider the case $ad-bc \equiv 0(p)$.
Using the estimate \eqref{bound on W2 xi eta} to bound $W_2$ along with the Weil bound in \eqref{Voronoi Transformed SCP}, we get
\begin{equation}
\begin{split}
S^{1,1,1} \ll (abpM_1M_2)^\epsilon\frac{(M_1M_2)^{3/4}} {abp^2Q^2}\sum_{\substack{q \leq Q} }{\frac{1}{\sqrt{q}}}\sum_{\substack{n_1 \leq N_1 \\ n_2 \leq N_2} }{\frac{|\lambda_f(n_1)|}{n_1^{1/4}}\frac{|\lambda_g(n_2)|}{n_2^{1/4}} } \times \\
(n_1,a)^{1/2}(n_2,b)^{1/2}(ad-bc,q)^{1/2}(bn_1-an_2,p) +O((pM_1M_2)^{-2018)}).
\end{split}
\end{equation}
Isolating the biggest common divisor of $(\text{Rad}(a),n_1)$ and $(\text{Rad}(b),n_2)$, we get
\begin{align}\label{S 1,1,1 bracket}
S^{1,1,1} &\ll (abpX)^\epsilon\frac{(M_1M_2)^{3/4}} {abp^2Q^2}\sum_{\substack{d_3|(ad-bc)\\ q \leq Q/d_3 } }{\frac{1}{\sqrt{q}}}\sum_{d_1|a,d_2|b}{(d_1d_2)^{1/2} }\sum_{\substack{d_1|d_1'|\text{Rad}(d_1)\\d_2|d_2'|\text{Rad}(d_2)} }{\frac{|\lambda_f(d_1')|}{d_1'^{1/4}} \frac{|\lambda_f(d_2')|}{d_2'^{1/4}} } \nonumber \\ 
&\left(p\sum_{\substack{n_1 \ll \frac{N_1}{d_1'} \\ n_2 \ll \frac{N_2}{d_2'} \\ bd_1'n_1-ad_2'n_2 \equiv 0(p)} }{\frac{|\lambda_f(n_1)|}{n_1^{1/4}}\frac{|\lambda_g(n_2)|}{n_2^{1/4}} } + \sum_{\substack{n_1 \ll \frac{N_1}{d_1'} \\ n_2 \ll \frac{N_2}{d_2'} } }{\frac{|\lambda_f(n_1)|}{n_1^{1/4}}\frac{|\lambda_g(n_2)|}{n_2^{1/4}}} \right).
\end{align}
We use Lemma \ref{Bounds for convolution of fourier coefficients on AP}  with $X =\frac{N_1}{d_1}$ , $Y = \frac{N_2}{d_2}$ and $c=p$, to bound the first term inside the bracket (see equation \eqref{definition of N1 and N2} for the definition of $N_1$ and $N_2$). This gives,
\begin{align}
S^{1,1,1} &\ll (abpX)^\epsilon\frac{(M_1M_2)^{3/4}} {abp^2Q^2} \times (N_1N_2)^{3/4}\sum_{q \leq Q}{\frac{1}{\sqrt{q}}}\\
&\ll (abpX)^\epsilon (K_1K_2)^{3/2}\sqrt{ab}p^{1/4}X^{3/4}.
\end{align}

The right hand side of the inequality above also bounds the second term in the bracket of equation \eqref{S 1,1,1 bracket}.  This proves the bound claimed in equation \eqref{bound for SCP shift multiple of level}.

If $(ad-bc,p)=1$, then the proof works out identically till equation \eqref{Voronoi Transformed SCP}. But we have a Kloosterman modulo $p$ instead of the Ramanujan sum now. Thus, using the Weil bound, we won't have the first term in the bracket of \eqref{S 1,1,1 bracket} and the second term would be multiplied by $\sqrt{p}$. Hence, the upper bound of equation \eqref{bound for SCP non-zero shift} is multiplied by an additional $\sqrt{p}$.

We had restricted our attention to the terms satisfying gcd of $(\alpha,apq)=(\beta,bpq)=(q,ab)=1$ in \eqref{Transformed SCP}:
{\small
\begin{equation}
\begin{split}
S= \frac{c_0}{abpQ}\sum_{q=1}^{\infty}{\frac{1}{q}}\sum_{ \substack{\alpha(apq) \\ \beta(bpq) \\ \alpha \equiv \beta (pq) \\ (\alpha,q)=1} }\sum_{\substack{n_1 \\n_2 }}{\lambda_f(n_1)\lambda_g(n_2)}{e\left(\frac{\alpha (n_1 -c)}{apq} -\frac{\beta (n_2 -d)}{bpq}   \right)}W_1\left(\frac{n_1}{M_1},\frac{n_2}{M_2} \right).
\end{split}
\end{equation}
}
We now indicate the modifications necessary to handle the general case. Let us first consider the contribution of the frequencies $\alpha$ which are divisible by $p$, denoted by $S^{\alpha \equiv 0 (p)}$. Since $\beta \equiv \alpha (p)$, this implies $\beta \equiv 0 (p)$. Furthermore, since $(\alpha,q)=1$ and $p$ divides $\alpha$, $(q,p) =1$. We decompose $a=a_1a_2$, $b=b_1b_2$, and $q = q_1q_2$ such that $(\alpha,ap) = a_1p$, $(\beta,bp)=b_1p$, and $(q, \text{Rad}(ab))= q_1$ respectively. Then
\begin{equation}\label{S alpha divisible by p}
S^{\alpha \equiv 0 (p)} = \sum_{a_1|a}\sum_{b_1|b}\sum_{\substack{q_1| \text{Rad}(ab) \\ q_1 \leq Q} }S^{a_1,b_1,q_1},
\end{equation}
where
{\small
\begin{equation}
\begin{split}
S^{a_1,b_1,q_1}= \frac{c_0}{abpQ}\sum_{(q_2,ab)=1 }{\frac{1}{q_1q_2}}&\sum_{ \substack{\alpha(a_2q_1q_2)^* \\ \beta(b_2q_1q_2)^* \\ a_1\alpha \equiv b_1\beta (q_1q_2)} } \\
&\sum_{\substack{n_1 \\n_2 }}{\lambda_f(n_1)\lambda_g(n_2)}{e\left(\frac{\alpha (n_1 -c)}{a_2q_1q_2} -\frac{\beta (n_2 -d)}{b_2q_1q_2}   \right)}W_1\left(\frac{n_1}{M_1},\frac{n_2}{M_2} \right),
\end{split}
\end{equation}
}
{\small
\begin{equation}
W_1\left(\frac{n_1}{M_1},\frac{n_2}{M_2}\right) = h\left(\frac{q_1q_2}{Q},\frac{\frac{n_1-c}{ap}-\frac{n_2-d}{bp}}{Q^2} \right) W\left(\frac{n_1}{M_1},\frac{n_2}{M_2}\right).
\end{equation}
}

We apply Voronoi summation \eqref{voronoi summation} formula to the $n_1$ and $n_2$ sum. The only difference in this case is that moduli ($a_2q_1q_2$ and $b_2q_1q_2$) are not divisible by $p$. The Voronoi transformed sum is
{\small
\begin{equation}\label{S a,b,q}
\begin{split}
S^{a_1b_1,q_1}&= \frac{c_0}{abpQq_1} \sum_{ (q_2,a_2b_2)=1 }{\frac{1}{q_2}}\sum_{ \substack{\alpha(a_2q_1q_2)^* \\ \beta(b_2q_1q_2)^* \\ a_1\alpha \equiv b_1\beta (q_1q_2)} }\\
&\sum_{\substack{n_1 \\n_2 }}{\overline{\lambda_f}(n_1) \overline{\lambda_g}(n_2) }{e\left(-\frac{\overline{p \alpha} n_1}{a_2q_1q_2} +\frac{\overline{p\beta }n_2}{b_2q_1q_2}   \right) e\left(-\frac{\alpha c}{a_2q_1q_2} +\frac{\beta d}{b_2q_1q_2}  \right)}W_2\left(\frac{M_1n_1}{p(a_2q_1q_2)^2},\frac{M_2n_2}{p(b_2q_1q_2)^2} \right)\\
& + \sum_{\substack{(-,+)} }{\dots} +\sum_{\substack{(+,-)} }{\dots} +\sum_{\substack{(-,-)} }{\dots},
\end{split}
\end{equation}
}
where
\begin{equation}
W_2\left(\xi,\eta\right)= \frac{M_1M_2}{a_2b_2p(q_1q_2)^2}\int_0^\infty\int_0^\infty{W_1(x,y)J^+_f\left(4\pi \sqrt{\xi x}\right) J^+_g\left(4 \pi \sqrt{\eta y}\right)dxdy},
\end{equation}
and the summation $\sum_{\substack{ (-,+) \\ (+,-) \\ (-,-)}}{\dots}$, refers to the three other terms coming from the $\pm$ terms on the right hand side of the Voronoi formula \eqref{voronoi summation}. Like before, they have the same structure and can be handled similarly. We use the Chinese remainder theorem to rewrite the exponential sums as Kloosterman sums. Then the sum over $\alpha(a_2q_1q_2)$ and $\beta(b_2q_1q_2)$ satisfying $ a_1\alpha \equiv b_1\beta (q_1q_2)$,  can be rewritten as
{\small
\begin{multline}
\sum_{\substack{\alpha (a_2q_1q_2) \\ \beta(b_2q_1q_2) \\ a_1\alpha \equiv b_1\beta(q_1q_2) } }{e\left(-\frac{\overline{p \alpha} n_1}{a_2q_1q_2} +\frac{\overline{p\beta }n_2}{b_2q_1q_2}   \right) e\left(-\frac{\alpha c}{a_2q_1q_2} +\frac{\beta d}{b_2q_1q_2}  \right)} = \\ S(\overline{abq_1}(ad-bc);\overline{b_2a_2q_1p}(a_2b_1n_2-a_1b_2n_1) ;q_2)\\ \frac{1}{q_1}\sum_{\delta(q_1)}{S(\delta a_1 -\overline{q_2}c; -\overline{pq_2}n_1;a_2q_1 )S( \overline{q_2}d - \delta b_1; \overline{pq_2}n_2;b_2q_1 )}.
\end{multline}
}
Using the Weil-bound for Kloosterman sum \eqref{Weil bound}, this is bounded by
$$
(a_2b_2q_2q_1^2)^{1/2+ \epsilon} (ad-bc,q_2)^{1/2} (n_1,a_2q_1)^{1/2}(n_2,b_2q_1)^{1/2}.
$$ 
Substituting this bound into \eqref{S a,b,q}, using the Rankin-Selberg bound \eqref{rankin selberg bound}, we get
$$
S^{a_1,b_1,q_1} \ll  (abpX)^\epsilon (K_1K_2)^{3/2}\sqrt{ab}q_1^{-1/2}p^{-3/4}X^{3/4}.
$$
This implies
$$
S^{\alpha \equiv 0(p)} \ll (abpX)^\epsilon (K_1K_2)^{3/2}\sqrt{ab}p^{-3/4}X^{3/4},
$$
which is stronger than the one claimed in \eqref{bound for SCP shift multiple of level}.

We are left with the contribution of frequencies $\alpha$ such that $(\alpha,p)=1$ to \eqref{Voronoi Transformed SCP}, denoted by $S^{(\alpha,p)=1}$. 
If $\alpha \not\equiv 0(p) $, then $\beta \not\equiv 0(p)$, as $\alpha \equiv \beta(pq)$. Let us decompose $a =a_1a_2$, $b=b_1b_2$, and $q=q_1q_2$ such that $(\alpha,apq) = a_1$, $(\beta,bpq)=b_2$, and  $(q, \text{Rad}(ab)) = q_1$ respectively. Then
{\small
\begin{equation}
\begin{split}
S^{(\alpha,p)=1}= \frac{c_0}{abpQ}\sum_{q_1| \text{Rad}(ab)}&\sum_{(q_2,ab)=1}^{\infty}{\frac{1}{q_1q_2}}\sum_{\substack{a_1|a\\ b_1|b} }\sum_{ \substack{\alpha(a_2pq_1q_2)* \\ \beta(b_2pq_1q_2)* \\ a_1\alpha \equiv b_1\beta (pq_1q_2)} }\\
&\sum_{\substack{n_1 \\n_2 }}{\lambda_f(n_1)\lambda_g(n_2)}{e\left(\frac{\alpha (n_1 -c)}{a_2pq_1q_2} -\frac{\beta (n_2 -d)}{b_2pq_1q_2}   \right)}W_1\left(\frac{n_1}{M_1},\frac{n_2}{M_2} \right).
\end{split}
\end{equation}
}
We apply Voronoi summation to the $n_1$ and $n_2$ sum modulo $a_2pq_1q_2$ and $b_2pq_1q_2$ respectively, and proceed exactly like the proof above for $S^{1,1,1}$.
\end{proof}

We now consider the case $p|(a,b)$. 
\begin{theorem}\label{SCP in a,b divisible by p}
Let $p$ be a prime number or $p=1$. Let  $a,b,c,d$ be integers such that $a$ and $b$ are non-zero and $p$ divides gcd$(a,b)$. Let $f,g$ be non-exceptional cuspidal newforms (modular or Maass) of level p and trivial nebentypus. For any $M_1,M_2,K_1,K_2 \geq 1$, we claim the following upper bounds for the shifted convolution sum $S_{f,g}$:
\begin{equation}
S_{f,g}(a,b,c,d,M_1,M_2) \ll p^\epsilon \min \{(M_1M_2)^{1/2},(M_1M_2)^\theta X\},
\end{equation}
where $X$ has been defined in \eqref{defintion of X} and $\theta$ is the exponent towards Ramanujan-conjecture for $f$ and $g$. Further, if the shift $ad-bc$ is non-zero, then
\begin{equation}\label{bound for SCP non-zero shift ab divisible by p}
S_{f,g}(a,b,c,d,M_1,M_2) \ll p^\epsilon(K_1K_2)^{3/2}\sqrt{ab}X^{3/4}.
 \end{equation}
\end{theorem}

\begin{proof}
The proof works exactly like the previous one. We just skip the step of factoring the delta symbol through a congruence mod $p$ i.e \eqref{congruence mod p}. Let $Q = \sqrt{X}$. We use the following expression
{\small
 \begin{equation}
 \begin{split}
\delta\left(\frac{n_1-c}{a}=\frac{n_2-d}{b} \right) = \frac{c_0}{Q}\sum_{q=1}^{\infty}{\frac{1}{q}}&\sum_{\gamma (q)^*}e\left(\frac{\gamma (n_1 -c)}{aq} -\frac{\gamma (n_2 -d)}{bq}   \right)h\left(\frac{q}{Q},\frac{\frac{n_1-c}{a}-\frac{n_2-d}{b}}{Q^2} \right),
 \end{split}
 \end{equation}
}
instead of \eqref{factored delta symbol} and proceed identically from here.
\end{proof}

\begin{theorem}\label{SCP square root saving}
Let $p$ be fixed. Let  $a,b,c,d$ be integers such that $a$ and $b$ are positive.  Further, let $f,g$ be non-exceptional newforms of level p and any nebentypus.

For any $M_1,M_2,K_1,K_2 \geq 1$, we claim the following upper bounds for the shifted convolution sum $S_{f,g}$:
\begin{equation}
S_{f,g}(a,b,c,d,M_1,M_2) \ll_{p} \min \{(M_1M_2)^{1/2},(M_1M_2)^\theta X\},
\end{equation}
where $X$ has been defined in \eqref{defintion of X} and $\theta$ is the exponent towards Ramanujan-conjecture for cusp forms on congruence subgroups of $SL_2(\mathbb{Z})$.
If the shift $ad-bc$ is non-zero then,
\begin{equation}\label{bound for SCP non-zero shift square root saving}
S_{f,g}(a,b,c,d,M_1,M_2) \ll_{p,K_1,K_2}  (abX)^{1/2+\theta}
 \end{equation}
 where the dependence on $p$ is polynomial.
\end{theorem}

\begin{proof}
Theorem 1.3 of \cite{blomer2004shifted} shows that if the shift $ad-bc \neq 0$ and $a,b$ are coprime, then
$$
S_{f,g}(a,b,c,d,M_1,M_2) \ll_{p,K_1,K_2} (bM_1 + aM_2)^{1/2 +\theta} \ll (abX)^{1/2+\theta} ,
$$
where $\theta$ is the exponent towards Ramanujan conjecture for cusp forms on congruence subgroups and $X$ is as in \eqref{defintion of X}. Using a minor modification we can handle the case $(a,b) >1$. The basic idea is to handle the sum over $q$ in \eqref{Voronoi Transformed SCP} using spectral theory for $GL(2)$. (Blomer \cite{blomer2004shifted} carries this out using Jutila's circle method.) 
\end{proof}
\section{Proof of Theorem \ref{main theorem} and Theorem \ref{main theorem when f is fixed}}\label{proof of main proposition}

Let us split  \eqref{S_p after cauchy} into the Diagonal part $\mathcal{D}$ \eqref{Diagonal contribution} (terms with zero shift) and the Off-Diagonal part $\mathcal{O}$ \eqref{Off diagonal} (terms with non-zero shift):
\begin{equation}\label{splitting into diagonal and off diagonal}
(S_{\vec{p} })^2 \ll \mathcal{D}_{\vec{p}} + \mathcal{O}_{\vec{p}} + O((pq)^{-1000})
\end{equation}
Note that the shift is $p_1p_3q(rt' - r't)$. We first consider the contribution of terms with $0$ shift (see \ref{definition of SCP} for definition of shift) in \eqref{S_p after cauchy}.
{\small
\begin{equation}\label{Diagonal contribution}
\begin{split}
\mathcal{D}_{\vec{p}} &=(pq)^\epsilon \left(\frac{N}{pq}\right)^4 \frac{p_1}{(p_2|\mathcal{L}| |\mathcal{C}| )^2}\sup_{z}\sum_{R =2^\nu \leq L^2}\sum_{\substack{(\rho_1, A \leq (pq)^\epsilon X_{\vec{p}}) \\ (\rho_2, B \leq (pq)^\epsilon Y_{\vec{p}}) } }{\frac{R^2B}{C^4} } \sum_{r,r' \in [R,2R]}{ \gamma_r \overline{\gamma_{r'}} } \sum_{\substack{\ |t|, |t'| \neq 0 \leq T_{\vec{p}} \\ rt' = r't} }{\beta_t \beta_{t'}}\\
&\left(\sum_{\substack{m}}{\overline{\lambda_f}(p_1rm+ p_3qt)\lambda_f(p_1r'm+ p_3qt') } W_{2,z}\left(\frac{p_1rm+ p_3qt}{A}\right)\overline{W_{2,z}\left(\frac{p_1r'm+ p_3qt'}{A}\right) } \right).
\end{split}
\end{equation}
}

If we assume the Ramanujan conjecture for Fourier coefficients of $f$, then the bound claimed in the lemma below is straightforward. We have to be a little careful as we want to treat this using Rankin-Selberg bounds alone.
\begin{lemma}\label{Bounds on diagonal contribution}
For all factorizations $\vec{p} = (p_1,p_2,p_3)$ of $p$, if $\mathcal{D}_{\vec{p}}$ is defined as in \eqref{Diagonal contribution}, then
\begin{equation}
\mathcal{D}_{\vec{p}} \ll (pq)^\epsilon\frac{N (pq)}{L}
\end{equation}

\end{lemma}

\begin{proof}
Recall \eqref{definition of Xp,Yp,Sp} and \eqref{definition of Tp}:
\begin{equation}
X_{\vec{p}} = \frac{p_1(p_2p_3qc)^2}{N}, Y_{\vec{p}} = \frac{(p_2p_3qc)^2}{rN}, \text{ and } T_{\vec{p}} = \max\left\{\frac{A}{p_3q}, \frac{p_1r B}{p_3q} \right\}.
\end{equation} 
 $\gamma_r$ is non-zero only when $r$'s are primes or squares of primes and we have dyadically divided the $r$ sum. Thus if $r \neq r'$, then $(r,r')=1$. Hence $rt' = r't$ implies that one of the following holds:
\begin{align}
&r=r', t=t'\\
&r \neq r', t= rk , t' =r'k\label{shifts case 2}
\end{align}
In the first case (we shall call this $\mathcal{D}_1$),
{\small
\begin{equation}\label{main diagonal contribution}
\begin{split}
\mathcal{D}^1_{\vec{p}} &\ll (pq)^\epsilon \left(\frac{N}{pq}\right)^4 \frac{p_1}{(p_2|\mathcal{L}| |\mathcal{C}| )^2}\sup_{z}\sum_{R =2^\nu \leq L^2}\sum_{\substack{(\rho_1, A \leq (pq)^\epsilon X_{\vec{p}}) \\ (\rho_2, B \leq (pq)^\epsilon Y_{\vec{p}}) } }{\frac{R^2B}{C^4} } \sum_{r \in [R,2R]}{ |\gamma_r|^2  } \sum_{\substack{\ |t| \leq T_{\vec{p}} } }{|\beta_t|^2 }\\
&\left(\sum_{\substack{m}}{|\lambda_f(p_1rm+ p_3qt)|^2 } \left| W_{2,z}\left(\frac{p_1rm+ p_3qt}{A}\right)\overline{W_{2,z}\left(\frac{p_1rm+ p_3qt}{A}\right) }\right|  \right).
\end{split}
\end{equation}
}
As $m$ and $t$ vary in the inner sum, $n= rp_1m + tp_3q$ runs over numbers smaller than $A$. Since $(rp_1, p_3q)=1$
 $$
rp_1 m_1+ t_1p_3q = rp_1 m_2 + t_2p_3q
 $$
 then $m_1 = m_2 + up_3q$ and $t_1 = t_2 -urp_1$. Thus the multiplicity of any $n$ is at most $O(1+ \frac{B}{p_3q} + \frac{A}{p_1p_3rq})$.  We recall from \eqref{definition of gamma and beta} that $\beta_t \ll d(t)$ and $\gamma_r \leq 2|\alpha_r|$. Using the Rankin-Selberg bound \eqref{rankin selberg bound} to bound the sum over $m$ and $t$, we get
 \begin{align}\label{bound for main diagonal}
\mathcal{D}^1_{\vec{p}} &\ll (pq)^\epsilon \left(\frac{N}{pq}\right)^4 \frac{p_1}{(p_2|\mathcal{L}| |\mathcal{C}| )^2} \times \sum_{R = 2^\nu\leq L^2}{ \frac{R^2}{C^4} }\sum_{r \in [R,2R]}{|\gamma_r|^2} Y_{\vec{p}}X_{\vec{p}}\left(1 + \frac{Y_{\vec{p}} }{p_3q} \right)\\
&\ll (pq)^\epsilon \frac{N (pq)}{L}.
 \end{align}
For the second case  of zero shift i.e \eqref{shifts case 2}:
{\small
\begin{equation}
\begin{split}
\mathcal{D}^2_{\vec{p}} &=(pq)^\epsilon \left(\frac{N}{pq}\right)^4 \frac{p_1}{(p_2|\mathcal{L}| |\mathcal{C}| )^2}\sup_{z}\sum_{R =2^\nu \leq L^2}\sum_{\substack{(\rho_1, A \leq (pq)^\epsilon X_{\vec{p}}) \\ (\rho_2, B \leq (pq)^\epsilon Y_{\vec{p}}) } }{\frac{R^2B}{C^4} } \sum_{r \neq r' \in [R,2R]}{ \gamma_r \overline{\gamma_{r'}} } \sum_{\substack{\ |k| \leq 2T_{\vec{p}}/R } }{\beta_{rk} \beta_{r'k}}\\
&\left(\sum_{\substack{m}}{\overline{\lambda_f}(r(p_1m+ p_3qk))\lambda_f(r'(p_1m+ p_3qk) ) } W_{2,z}\left(\frac{r(p_1m+ p_3qk)}{A}\right)\overline{W_{2,z}\left(\frac{r'(p_1m+ p_3qk)}{A}\right) } \right).
\end{split}
\end{equation}
}
As $m$ and $k$ vary in the inner sum, $n= p_1m + kp_3q$ runs over number smaller than $(pq)^\epsilon X_{\vec{p}}/r$. Since $(p_1, p_3q)=1$
 $$
p_1 m_1 + k_1p_3q = p_1 m_2 + k_2p_3q
 $$
 then $m_1 = m_2 + tp_3q$ and $k_1 = k_2 -tp_1$. Thus the multiplicity of any $n$ is at most $O(1+ (pq)^\epsilon\frac{Y_{\vec{p}}}{p_3q})$. So, the sum over $m$ and $k$ is bounded by
\begin{align}
\sum_{m,k}{\dots} &\ll (pq)^\epsilon\left(1+ \frac{Y_{\vec{p}}}{p_3q}\right)\sum_{n\leq (pq)^\epsilon X_{\vec{p}}/r}{|\lambda_f(rn)||\lambda_f(r'n)|} \\
&\ll (pq)^\epsilon\left(1+ \frac{Y_{\vec{p}}}{p_3q}\right)|\sigma_f(r)||\sigma_f(r')|\frac{X_{\vec{p}}}{r},
 \end{align}
 where $\sigma_f$ has been defined in \eqref{definition of sigma f}. It is straight forward to verify the bound claimed in the Lemma for $\mathcal{D}_{\vec{p}}$ now.
\end{proof}

Having treated the zero-shift terms, we are left with the non-zero shifts $\mathcal{O}_{\vec{p}}$ :
{\small
\begin{equation}\label{Off diagonal}
\begin{split}
\mathcal{O}_{\vec{p}} & :=(pq)^\epsilon \left(\frac{N}{pq}\right)^4 \frac{p_1}{(p_2|\mathcal{L}| |\mathcal{C}| )^2}\sup_{z}\sum_{R =2^\nu \leq L^2}\sum_{\substack{(\rho_1, A \leq (pq)^\epsilon X_{\vec{p}}) \\ (\rho_2, B \leq (pq)^\epsilon Y_{\vec{p}}) } }{\frac{R^2B}{C^4} } \sum_{r,r' \in [R,2R]}{ \gamma_r \overline{\gamma_{r'}} } \sum_{\substack{\ |t|, |t'| \neq 0 \leq T_{\vec{p}} \\ rt' \neq r't} }{\beta_t \beta_{t'}}\\
&\left(\sum_{\substack{m}}{\overline{\lambda_f}(p_1rm+ p_3qt)\lambda_f(p_1r'm+ p_3qt') } W_{2,z}\left(\frac{p_1rm+ p_3qt}{A}\right)\overline{W_{2,z}\left(\frac{p_1r'm+ p_3qt'}{A}\right) } \right).
\end{split}
\end{equation}
}
We have to treat the cases $p_1=p$, $p_2=p$ and $p_3=p$ separately. We plan to bound the inner sum over $m$ by applying Theorem \ref{SCP in a,b divisible by p} if $p_1=p$ and Theorem \ref{SCP in ab coprime to p} otherwise. 

\begin{lemma}\label{O p_1 =p}
Let $\mathcal{O}_{(p,1,1)}$ be defined as in \eqref{Off diagonal}. Then
\begin{equation}
\mathcal{O}_{(p,1,1)} \ll (pq)^\epsilon\frac{N (pq) L^{25/4}}{p^{7/4}q^{1/4}}.
\end{equation}

\end{lemma}

\begin{proof}
If $\vec{p} = (p,1,1)$, then $X_{\vec{p}}= \frac{p(qc)^2}{N}$, $Y_{\vec{p}} = \frac{(qc)^2}{RN}$, and $T_{\vec{p}} = \max\left\{\frac{A}{q}, \frac{pr B}{q} \right\}$ (see \eqref{definition of Xp,Yp,Sp} and \eqref{definition of Tp}). Use Theorem \eqref{SCP in a,b divisible by p} with $a=pr$, $b=pr'$, $c=qt$, $d=qt'$, $M_1=M_2 =A$, and $K_1=K_2=1$ to bound the $m$-sum. Then,
$$
\sum_{\substack{m}}{\overline{\lambda_f}(prm+ qt)\lambda_f(pr'm+ qt') } W_{2,z}\left(\frac{prm+ qt}{A}\right)\overline{W_{2,z}\left(\frac{pr'm+ qt'}{A}\right) } \ll (pq)^\epsilon (pR)^{1/4}A^{3/4}.
$$
Recall from \eqref{definition of gamma and beta} that $\beta_t \ll d(t)$ and $\gamma_r \leq 2|\alpha_r|$. Rankin-Selberg bound implies that $\sum_{r}{|\gamma_r|} \ll (pq)^\epsilon L$. Taking absolute values and using the bound bound above in \eqref{Off diagonal}, we get
\begin{equation}
\mathcal{O}_{(p,1,1)} \ll (pq)^\epsilon\frac{N(pq) T^{3/4} L^{11/2}}{p^{7/4} q^{1/4}} \ll  (pq)^\epsilon\frac{N(pq)L^{25/4}}{p^{7/4} q^{1/4}},
\end{equation}
where $T \leq L$ has been defined in \eqref{definition of T}. (It is possible to remove the $T^{3/4}$ in the numerator, by choosing $C= L^2/T$ in \eqref{definition of moduli C}.)
\end{proof}

\begin{lemma}\label{O p_2=p}
Let $\mathcal{O}_{(1,p,1)}$ be defined as in \eqref{Off diagonal}.
Then
\begin{equation}
\mathcal{O}_{(1,p,1)} \ll (pq)^\epsilon\frac{\sqrt{p}N (pq)L^{25/4}}{q^{1/4}} .
\end{equation}
\end{lemma}
\begin{proof}
Use Theorem \eqref{SCP in ab coprime to p}, with $a=r$, $b=r'$, $c=qt$, $d=qt'$, $M_1= M_2 =A$, and $K_1=K_2=1$ to bound the $m$-sum in \eqref{Off diagonal}. The shift $(rt'-r't)$ may not be multiple of $p$. Hence, we have
$$
\sum_{\substack{m}}{\overline{\lambda_f}(rm+ qt)\lambda_f(r'm+ qt') } W_{2,z}\left(\frac{rm+ qt}{A}\right)\overline{W_{2,z}\left(\frac{r'm+ qt'}{A}\right) } \ll (pq)^\epsilon R^{1/4}(pA)^{3/4}.
$$
If $\vec{p} = (1,p,1)$, then $X_{\vec{p}}= \frac{(pqc)^2}{N}$, $Y_{\vec{p}} = \frac{(pqc)^2}{RN}$, and $T_{\vec{p}} = \max\left\{\frac{A}{q}, \frac{r B}{q} \right\}$ (see \eqref{definition of Xp,Yp,Sp} and \eqref{definition of Tp}). The bound claimed in the Lemma is a straightforward consequence of the bound mentioned above, proceeding along the lines of Lemma \eqref{O p_1 =p}.
\end{proof}

\begin{lemma}\label{O p_3 =p}
Let $\mathcal{O}_{(1,1,p)}$ be defined as in \eqref{Off diagonal}. Then
\begin{equation}
\mathcal{O}{(1,1,p)} \ll (pq)^\epsilon\frac{N (pq)L^{25/4}}{q^{1/4}}.
\end{equation}

\end{lemma}

\begin{proof}
Use Theorem \ref{SCP in ab coprime to p} with $a=r$, $b=r'$, $c=pqt$, $d=pqt'$, $M_1= M_2 =A$, and $K_1=K_2=1$ to bound the $m$-sum in \eqref{Off diagonal}. In this case the shift is divisible by $p$. This is the reason we save an additional $\sqrt{p}$ as compared to Lemma \ref{O p_2=p}.
\end{proof}

\begin{lemma}\label{Bound on O in case p is fixed}
If $p$ is fixed. Then for any factorization $p_1p_2p_3=p$,
\begin{equation}
\mathcal{O}_{\vec{p}} \ll_{p} \frac{NqL^{23/4}(qL^7)^\theta}{q^{1/2}},
\end{equation}
where $\theta$ is the bound towards Ramanujan conjecture for the congruence subgroup $\Gamma_0(ab)$.
\end{lemma}
\begin{proof}
We use Theorem \ref{SCP square root saving} with $a=p_1r$, $b=p_1r'$ $c=p_3qt$, $d=p_3qt'$, $M_1=M_2=A$, and $K_1=K_2=1$ to bound the $m$-sum in \eqref{Off diagonal}. Proceeding like before, we get the claim stated in the lemma.  Note that we save $q^{1/2}$ as opposed to $q^{1/4}$ in the previous Lemma. 
\end{proof}

\subsection{Proof of Theorem \ref{main theorem}}

\begin{proof}
We trace our steps starting from the beginning. Lemma \ref{amplification} shows that
$$
S(N) = S_1(N) + O((pq)^\epsilon \frac{N}{\sqrt{L}}).
$$
Using Equation \eqref{S1 in dual side} and the observation that all four choices of $\{\pm,\pm \}$ behave the same way, we have
$$
S(N) = S_1^{dual}(N) +  O((pq)^\epsilon \frac{N}{\sqrt{L}}).
$$
Eliminating the boundary terms in \eqref{S1 dual after eliminating boundary terms}, led us to
\begin{equation}\label{bound for S(N) in terms of diagonal and off diagonal}
S(N)= \sum_{p_1p_2p_3=p}{\mathcal{S}_{(p_1,p_2,p_3)}(N)} + O\left((pq)^\epsilon\frac{\sqrt{Npq}}{\sqrt{L}}\right).
\end{equation}
After Cauchy-Schwarzing, equation \eqref{splitting into diagonal and off diagonal} shows that
$$
S(N) \ll \sum_{p_1p_2p_3=p}{\mathcal{D}_{(p_1,p_2,p_3)}(N)^{1/2} +\mathcal{O}_{(p_1,p_2,p_3)}(N)^{1/2} } + O\left((pq)^\epsilon \frac{\sqrt{Npq} }{\sqrt{L}}\right).
$$
We use Lemma \ref{Bounds on diagonal contribution} to control the diagonal contribution. We shall use the bound from Lemma \ref{O p_2=p} for $\mathcal{O}^{p_2=p}$ when $p$ is small. Otherwise we shall use Lemma \ref{Bound on S_2^p_2=p} to get rid of the contribution of the terms with $p_2=p$. Thus using Lemma \ref{Bound on S_2^p_2=p}, \ref{O p_1 =p}, \ref{O p_2=p},  and \ref{O p_3 =p}, we obtain:
\begin{align}
S(N)^2 &\ll (pq)^\epsilon N(pq)\left(\frac{1}{L} +\frac{L^{25/4}}{q^{1/4}} + \min\left\{ \frac{\sqrt{p}L^{25/4}}{q^{1/4}} ,\frac{L^2}{p^2}\right\} \right).
\end{align}
If $p \leq L^{3/2}$, then we use the first bound in $\min\{ \dots \}$ and the second bound otherwise. Making the choice $L= q^{1/32}$, we get
\begin{align}
S(N)^2 &\ll (pq)^\epsilon\frac{N(pq)}{L}.
\end{align}
Thus
$$
S(N) \ll_{\epsilon,A} (pq)^\epsilon \sqrt{Npq} q^{-1/64}.
$$
This proves Proposition \ref{main proposition} and therefore, Theorem \ref{main theorem}.
\end{proof}

\subsection{Proof of Theorem \ref{main theorem when f is fixed}}

\begin{proof}
 Proceeding as in the previous proof, we get equation \eqref{bound for S(N) in terms of diagonal and off diagonal}. Using Lemma \ref{Bounds on diagonal contribution} and \ref{Bound on O in case p is fixed} to bound $\mathcal{D}$ and $\mathcal{O}$ respectively, we get
$$
S(N)^2 \ll_{p,\epsilon} q^\epsilon Nq\left( \frac{1}{L} + \frac{L^{23/4+7\theta}}{q^{1/2-\theta}}\right).
$$
Equating the two terms, we get 
\begin{equation}\label{choice of L in main theorem}
L = q^{\frac{2(1-2\theta)}{27+28\theta} }.
\end{equation}
Thus
$$
S(N)  \ll q^{\epsilon} \frac{\sqrt{Nq}}{q^{\frac{(1-2\theta)}{27+28\theta}}}.
$$
This proves Proposition \ref{main proposition when f is fixed} and therefore, Theorem \ref{main theorem when f is fixed}.

\end{proof}

\bibliographystyle{amsplain}
\bibliography{sub_convexity}

\end{document}